\documentclass[11pt]{article}
\usepackage[letterpaper, margin=1.0in]{geometry}

\usepackage{times}
\usepackage{microtype}
\usepackage{fullpage}

\usepackage[dvipsnames,svgnames]{xcolor} 
\usepackage{tikz}
\usetikzlibrary{shapes}
\usetikzlibrary{positioning}
\usetikzlibrary{plotmarks}

\usepackage{microtype}
\usepackage{graphicx}
\usepackage{caption}
\usepackage{booktabs} 
\usepackage{multirow}
\usepackage{adjustbox}

\usepackage{stmaryrd}
\usepackage{graphicx}
\usepackage{caption}
\usepackage{subcaption}

\usepackage{textcomp} 
\usepackage{quoting}
\usepackage{titletoc} 
\usepackage{fontawesome} 

\usepackage[colorlinks = true,
            linkcolor = blue,
            urlcolor  = blue,
            citecolor = blue,
            anchorcolor = blue]{hyperref}

\newlength{\continueindent}
\usepackage{etoolbox}

\usepackage{amsmath}		
\usepackage{amssymb}		
\usepackage{amsfonts}		
\usepackage{amsthm}		    
\usepackage{mathtools}
\usepackage{caption}

\usepackage{mathrsfs}
\usepackage{nccmath}

\usepackage{stmaryrd}
\usepackage{listings}
\usepackage{color}

\usepackage{mathtools}		

\usepackage[utf8]{inputenc}		
\usepackage[T1]{fontenc}		

\usepackage{url}

\usepackage{dsfont}		

\usepackage{inconsolata}

\usepackage[
cal=cm,
]
{mathalfa}

\colorlet{MyBlue}{DodgerBlue!75!Black}
\colorlet{MyGreen}{DarkGreen!85!Black}

\usepackage{placeins}

\usepackage{wasysym}		

\usepackage{fancybox}		
\usepackage{tikz}		
\usetikzlibrary{calc,patterns}
\usepackage{pgfplots}
\pgfplotsset{compat=1.16}

\usepackage{array}		
\newcolumntype{C}[1]{>{\centering\let\newline\\\arraybackslash\hspace{0pt}}m{#1}}
\usepackage{booktabs}		
\usepackage[inline,shortlabels]{enumitem}		
\usepackage{multirow}
\usepackage{threeparttable}   

\usepackage{acronym}		
\usepackage{latexsym}		
\usepackage{xfrac}		
\usepackage{xspace}		
\usepackage{xparse}     

\usepackage{hyperref}
\hypersetup{
colorlinks=true,
linktocpage=true,
pdfstartview=FitH,
breaklinks=true,
pdfpagemode=UseNone,
pageanchor=true,
pdfpagemode=UseOutlines,
plainpages=false,
bookmarksnumbered,
bookmarksopen=false,
bookmarksopenlevel=1,
hypertexnames=true,
pdfhighlight=/O,
urlcolor=Maroon,linkcolor=MyBlue!60!black,citecolor=MyBlue!70!black,	
pdftitle={},
pdfauthor={},
pdfsubject={},
pdfkeywords={},
pdfcreator={pdfLaTeX},
pdfproducer={LaTeX with hyperref}
}
\usepackage[sort&compress,capitalize,nameinlink]{cleveref}		
\crefname{assumption}{Assumption}{Assumptions}
\crefname{assumptionloc}{Assumption}{Assumptions}

\usepackage{autonum}		

\usepackage[textwidth=20mm]{todonotes}		
\newtheorem{thm}{Theorem}
\newtheorem{prop}[thm]{Proposition}
\newtheorem{lem}[thm]{Lemma}
\newtheorem{cor}[thm]{Corollary}		
\newtheorem{example}{Example}[section]
\newtheorem{remark}{Remark}

\newenvironment{acknowledgements}
    {\large\bfseries Acknowledgement%
    \par\medskip\normalfont\normalsize}%
    {}%

\newtheorem{assumption}{Assumption}		

\newtheorem*{definition*}{Definition}		
\newtheorem*{assumption*}{Assumptions}		

\newtheorem{fw}{Framework}

\makeatletter		
\newcommand{\fwtag}[1]{
  \let\oldthefw\thefw
  \renewcommand{\thefw}{#1}
  \g@addto@macro\endfw{
    \addtocounter{fw}{-1}
    \global\let\thefw\oldthefw}
  }
\makeatother

\newcommand{\RR}{\mathbb{R}}

\usepackage[numbers,sort&compress]{natbib}
\bibpunct[, ]{[}{]}{,}{n}{,}{,}

\definecolor{dkgreen}{rgb}{0,0.6,0}
\definecolor{gray}{rgb}{0.5,0.5,0.5}
\definecolor{mauve}{rgb}{0.58,0,0.82}

\usepackage[linesnumbered,algoruled,boxed,lined]{algorithm2e}
\SetKwInput{KwInit}{Initialize}
\usepackage{hyperref}

\usepackage{color}
\definecolor{deepblue}{rgb}{0,0,0.5}
\definecolor{deepred}{rgb}{0.6,0,0}
\definecolor{deepgreen}{rgb}{0,0.5,0}

\usepackage{listings}
\usepackage{verbatim}

\newcommand\pythonstyle{\lstset{
language=Python,
basicstyle=\ttm,
otherkeywords={self},             
keywordstyle=\ttb\color{deepblue},
emph={MyClass,__init__},          
emphstyle=\ttb\color{deepred},    
stringstyle=\color{deepgreen},
frame=tb,                         
showstringspaces=false            %
}}

\DeclareMathOperator*{\argmin}{arg\,min}

\DeclareMathOperator*{\conv}{conv}
\DeclareMathOperator*{\jaco}{\nabla}

\newcommand{\x}{w}
\newcommand{\X}{u}
\newcommand{\D}{D}
\newcommand{\gnui}{{g_{\nu}}}

\newcommand{\gnu}{g_{\nu}} 
\newcommand{\qnu}{q_{\nu}} 
\newcommand{\barq}{\bar{q}}
\newcommand{\mnui}{m_{\nu}}

\newcommand{\etaa}{\underline{\eta}}
\newcommand{\etab}{\bar{\eta}}
\newcommand{\etas}{{\eta^{\star}}}

\newcommand{\Qp}{Q_p}
\newcommand{\Sp}{\mathbb{S}_p}
\newcommand{\Spnu}{\mathbb{S}^\nu_p}
\newcommand{\Loss}{L}

\newcommand{\fnu}{f_{\nu}}

\newcommand{\Rp}{S^p} 
\newcommand{\Rn}{R_n}
\newcommand{\Rpn}{S^p_n}
\newcommand{\bRpn}{\bar S^p_n}

\newcommand{\Pn}{P_n}

\DeclareMathOperator{\bigO}{O}
\DeclareMathOperator{\smallo}{o}

\newcommand{\vertiii}[1]{{\left\vert\kern-0.25ex\left\vert\kern-0.25ex\left\vert #1
    \right\vert\kern-0.25ex\right\vert\kern-0.25ex\right\vert}}
\newcommand{\trans}[1]{#1^\top}

\lstnewenvironment{python}[1][]
{
\pythonstyle
\lstset{#1}
}
{}

\newcommand\pythoninline[1]{{\pythonstyle\lstinline!#1!}}

\DeclareMathOperator*{\R}{\mathbb{R}}

\DeclareMathOperator*{\PP}{\mathbb{P}}
\newcommand{\Rd}{\mathbb{R}^d}

\renewcommand{\epsilon}{\varepsilon}
\newcommand{\dd}{\mathrm{d}}
\newcommand{\covnum}{N}
\newcommand \dist {\operatorname{dist}} 

\newcommand \T {^{\top}}	
\DeclarePairedDelimiterX{\inp}[2]{\langle}{\rangle}{#1, #2} 

 \newcommand{\w}{w}

\newcommand{\nfl}{m}
\newcommand{\nfair}{G}

\newcommand{\eg}{e.g.,\xspace}		
\newcommand{\ie}{i.e.,\xspace}		

\title{Superquantiles at Work: Machine Learning Applications and Efficient Subgradient Computation}

\author{
Yassine Laguel$^{1}$, Krishna Pillutla$^{2}$, Jérôme Malick$^{1}$, Zaid Harchaoui$^{2}$ \\
\small{$^{1}$Univ. Grenoble Alpes, CNRS, Grenoble INP, LJK, 38000 Grenoble, France}  \\
\small{$^{2}$University of Washington, Seattle, WA, USA}}

\begin{document}
\maketitle



\begin{abstract}
R.\;Tyrell Rockafellar and  collaborators introduced, in a series of works, new regression modeling methods based on the notion of superquantile (or conditional value-at-risk). These
methods
have been influential in economics, finance, management science, and operations research in general. Recently, they have been the subject of a renewed interest in machine learning, to address issues of distributional robustness and fair allocation. In this paper, we review some of these new applications of the superquantile, with references to recent developments. These applications involve nonsmooth superquantile-based objective functions that admit explicit subgradient calculations. To make these superquantile-based functions amenable to the gradient-based algorithms popular in machine learning, we show how to smooth them by infimal convolution and describe numerical procedures to compute the gradients of the smooth approximations.
We put the approach into perspective by comparing it to other smoothing techniques and by illustrating it on toy examples.
\end{abstract}


\section{Introduction}

\subsection{Superquantiles at Work: Old and New}

Risk measures play a crucial role in optimization under uncertainty, involving problems with an aversion to worst-cases scenarios. Among popular convex risk measures, the superquantile -- also called the Conditional Value at Risk, Tail Value at Risk, Mean Excess Loss, or Mean Shortfall -- has received special attention. The superquantile has been extensively studied from a convex analysis perspective: we refer, for instance, to \cite{rockafellar2000optimization} for a variational formulation of the superquantile, to \cite{zbMATH05213441} for its generalization to a larger class of risk measures, to \cite{DBLP:journals/fs/FollmerS02} for a dual formulation (also later generalized in \cite{ruszczynski2006optimization} or \cite{rockafellar2002conditional}) and \cite{rockafellar2014random} for additional convex properties. The superquantile can be traced back to the paper~\cite{zbMATH04016602}.
These nice theoretical properties have given interesting results in various applications, ranging from finance \cite{sarykalin2008value} to energy planning \cite{DBLP:journals/mp/GuiguesS13}; for a thorough discussion and many references, we refer to the seminal work\;\cite{rockafellar2000optimization}, the classical textbook \cite[Chap.\;6]{Shapiro:2014:LSP:2678054}, or the tutorial paper\;\cite{rockafellar2013superquantiles}.

More recently, the superquantile has also drawn an increasing attention in machine learning. In this paper, we give an overview of some of the new applications of the superquantile in machine learning: we discuss the use of the superquantile for distributionally robust learning, fairness in machine learning, federated learning, adversarial classification, and risk-sensitive reinforcement learning; we also give toy illustrations and pointers to recent exciting developments.

Superquantile optimization problems are nonsmooth, possibly non-convex, but also highly structured.
In financial or operations research applications,
these nonsmooth optimization problems are usually solved using one of two approaches: (a) extending specific algorithms (\eg progressive hedging for risk-averse multi-stage programming\;\cite{rockafellar2018solving}), or, (b) relying on convex programming (\eg linear programming coupled with Monte Carlo simulations for portfolio management\;\cite{rockafellar2000optimization}). We refer to\;\cite{rockafellar2014superquantile} and\;\cite{miranda2014superquantile} for discussions on computational approaches. In machine learning, recent papers propose to use
stochastic first-order optimization algorithms for superquantile learning; see \eg\cite{DBLP:conf/nips/CuriLJ020,levy2020large} and references therein.

In this paper, we propose a simple alternative.
We study the smoothing of superquantile by infimal-convolution, extending and clarifying the results of\;\cite[Sec.~3]{laguel2020first}. This opens the way for using first-order methods for smooth optimization: this is of special interest for machine learning applications where standard algorithms and software rely heavily rely on gradient-based optimization~\cite{DBLP:conf/osdi/AbadiBCCDDDGIIK16,paszke2017automatic}.
In fact, optimization guarantees in this context are typically given for smooth surrogates of the superquantile, e.g., \cite{laguel2020device,levy2020large}; which we study and clarify.
In view of these applications, we pay attention to provide
efficient procedures for computing gradients of smooth approximations of superquantile-based functions.
We illustrate these smoothed first-order oracles combined  with quasi-Newton methods on simple problems with synthetic or real data.
We refer to our recent work\;\cite{laguel2020first,laguel2020device} for more computational experiments, using particular cases of such efficient smoothed oracles.


More specifically, the contributions of this paper 
are multiple and can be pointed out, 
section by section, as follows:
\begin{itemize}
    \item We formalize, in Section\;\ref{sec:learning},  the existing notion of empirical superquantile minimization and provide a convergence result for supervised learning.
    \smallskip
    \item We propose, in Section\;\ref{sec:newapplis}, an overview of recent machine learning applications of superquantiles.
    \smallskip
    \item We study in Section\;\ref{sec:sub} the
    (sub)gradient calculus of superquantile-based functions with
    a focus on computational efficiency.
    In particular, Section\;\ref{sec:sub2} studies generalized subgradients of superquantile-based functions and
    Section\;\ref{sec:smooth} considers gradients of 
    smooth approximations of the superquantile by inf-convolution.
    Finally, we establish in Section\;\ref{sec:smooth-comp} the
    equivalence between different inf-convolution schemes, as well as the smoothing by convolution. We propose to use quasi-Newton algorithms to minimize these smoothed approximations.
\end{itemize}

\subsection{Superquantiles: Review and Notation}\label{sec:setting}


We recall basic definitions and properties used in this paper.
Our notation and terminology follow closely the ones of \cite{ruszczynski2006optimization} and \cite{rockafellar2013superquantiles}; we refer to these papers for more details and references.

Consider a probability space $\Omega$, with probability denoted $\PP$.
For $p \in (0,1)$, the $p$-quantile of a random variable $U\colon\Omega\rightarrow\RR$, denoted by $Q_p(U)$,
is the inverse of the cumulative distribution function of $U$. For all $t\in \RR$ we have 
\begin{equation}\label{eq:Qp}
\Qp(U) \leq t \iff \PP(U \leq t) \geq p\,.
\end{equation}
When \eg $p=1/2$, the $p$-quantile corresponds the median value of the random variable.
For $p \in [0,1)$, the $p$-superquantile of $U$ is defined as the mean of values of quantiles greater than the threshold $p$:
\begin{equation}\label{eq:def_cvar}
\Sp(U) = \frac{1}{1-p} \int_{p}^1 Q_{p^\prime}(U) \mathrm{d}p^\prime\,.
\end{equation}
The analogue to \eqref{eq:Qp} for the superquantile is stronger:
\[
\Sp(U) \leq  t \iff \text{$U$ is lower than $t$ on average in its $p$-tail.}
\]
The superquantile is thus interpreted as a measure of the upper tail of the distribution of $U$. Another interpretation comes from the
dual formulation of superquantiles\;\cite{DBLP:journals/fs/FollmerS02}:  $\Sp(U)$ can be written as a maximal expectation of $U$ with respect to probability measures having a (Radon-Nykodim) derivative bounded by $1-p$:
\begin{equation}\label{eq:def_max_cvar}
\Sp(U) = \max_{{\substack{0\leq q(\cdot) \leq \frac{1}{1-p}\\ \int_\Omega q\, d\,\mathbb{P}(\omega) = 1}}} \int_{\omega \in \Omega} U(\omega) q(\omega) \mathrm{d} \mathbb{P}(\omega) \,.
\end{equation}
When $U$ is a discrete random variable, the above expression simplifies; we come back to this in Section~\ref{sec:sub}. Finally, for an optimization perspective, the superquantile also has a nice variational formulation\;\cite{rockafellar2000optimization}:
\begin{equation}\label{eq:def_min_cvar}
\Sp(U) = \min_{\eta \in \mathbb{R}} \left\{\eta + \frac{1}{1-p} \mathbb{E}[\max(U-\eta,0)]\right\}.
\end{equation}
In this expression, the quantile $\Qp(U)$ is obtained as the left end-point of the solution.
This last expression also reveals an important advantage of $\Sp(U)$ over $\Qp(U)$ as a measure of the tail of $U$, from both theoretical and practical points of view: the superquantile is convex, positively homogeneous, monotonic, translation invariant; see, \eg, the tutorial article\;\cite{rockafellar2013superquantiles}.

\section{Standard and Superquantile Machine Learning}\label{sec:learning}

Optimization is at the heart of machine learning, through the paradigm of empirical risk minimization, which we briefly recall in Section\;\ref{sec:classical}.
In Section\;\ref{sec:safer},
we discuss superquantile learning,
where the risk measure of the learning model is the superquantile.
The material of this section also serves as a gentle introduction to the recent developments outlined in the next section.

\subsection{Supervised Learning Review}\label{sec:classical}

We recall here the notation and basic notions of supervised learning; we refer to standard textbooks\;\cite{zbMATH05133436} or \cite{shalev2014understanding} for more details.
In the training phase of supervised learning, we have access to $n$ data-points: each data-point is a pair $(x, y)$, where $x\in X$ is a feature vector and $y\in Y$ is its corresponding target.  For instance, for a binary classification task, $y$ is a Boolean encoding the membership of the image $x$ to one of the two classes. From this training data, the aim is to learn a parameter $\x\in W \subset \Rd$ as ``weights'' of a given prediction function $z = \varphi(\x,x)$ that produces, for an input $x\in X$, a prediction $z \in Z$ of the associated target $y\in Y$.
Typical examples of prediction functions 
include simple linear models $\varphi(\x,x)= \trans{\x}x$, polynomial models (as in Example\;\ref{ex:learning} below), or artificial neural networks
\begin{equation}\label{eq:NN}
\varphi(\x,x)= \trans{\x_s}\sigma(\cdots\sigma(\trans{\x_1}x))   \,,
\end{equation}
which are successive compositions of linear models $\x_j$ and non-linear activations $\sigma$.
The prediction error is then measured by a loss function $\ell: Y \times Z \rightarrow \R$. Typical examples of loss functions  include the least-squares loss ($Y\!=\R, Z\!=\R$) or the logistic loss ($Y=\{-1, 1\}, Z=\R$), defined respectively as
\begin{equation}\label{eq:losses}
\ell(y, z) = \frac{1}{2}(y-z)^2\,, \qquad\text{and,}\qquad
\ell(y, z) = \log(1+\exp(-y\,z)) \,.
\end{equation}

Assuming\footnote{When the assumption of the existence of an underlying distribution $P$ is not realistic, the usual approach is to still use the empirical risk minimization~\eqref{eq:ERM} from the given training dataset
$\Pn= \{(x_i,y_i)\}_{1 \leq i \leq n}$.} that the training data are generated from a given distribution $P$ over $X \times Y$, the ``best" model parameter $\x$ solves
the optimization problem
\begin{equation}\label{eq:pop-risk}
    \min_{\x \in W} \left[
        R(w) = ~\mathbb{E}_{(x,y)\sim P}
        \left[\ell(y, \varphi(\x,x)\right]
    \right]\,.
\end{equation}
However, we can only access $P$ via i.i.d. samples
$\{(x_i,y_i)\}_{1 \leq i \leq n}$. So we consider instead the empirical risk minimization approach,
which solves the following optimization problem, analogous to
\eqref{eq:pop-risk} but where the expectation is taken over $\Pn$, the empirical measure over the training examples:
\begin{equation}\label{eq:ERM}
\min_{\x \in W}
\left[
\Rn(w) =~\mathbb{E}_{(x,y)\sim\Pn}\left[\ell(y, \varphi(\x,x )\right]
= \frac{1}{n} \sum^n_{i=1}\ell(y_i, \varphi(\x,x_i )
\right] \,.
\end{equation}
Under suitable conditions, we have that the minimizer $\x_n^\star$ of \eqref{eq:ERM}
converges almost surely in mean error to the best population error as $n \to \infty$, i.e.,
\begin{equation}\label{eq:convergence}
R(w_n^\star)\underset{n \to \infty}{\longrightarrow} R(w^\star) \qquad\text{almost surely}.
\end{equation}

For concreteness, we instantiate this general framework with a simple regression task, which will also be used in illustrations in subsequent sections.

\begin{example}[Least-squares regression]\label{ex:learning}
Consider a dataset $\D=(x_i, y_i)_{1\leq i \leq n}\in (\RR\times\RR)^n$ generated by noisy observations of a quadratic function: we have
\begin{equation}\label{eq:data}
y_i = \bar\x_{0} + \bar\x_{1}\; x_i + \bar\x_{2}\; x_i^2 + \varepsilon_i \,,
\qquad\text{where }\varepsilon_i \sim \mathcal{N}(0,\sigma^2)\,,
\end{equation}
for an unknown vector $\bar\x = (\bar\x_0, \bar\x_1,\bar\x_2) \in \RR^3$ that we would like to approximate.
In this case, \eqref{eq:ERM} instantiates as the ordinary least-squares problem
\begin{equation}\label{eq:OLS}
    \min_{\x\in \mathbb{R}^3} \mathbb{E}_{(x, y) \sim \Pn}\big[(y - (\x_2 x^2 + \x_1 x + \x_0))^2\big] \,,
\end{equation}
with a quadratic model $\varphi(\x,\cdot)$ and the square loss.\qed
\end{example}

\subsection{Superquantile Learning}\label{sec:safer}

The standard framework, recalled above, is currently challenged by important domain applications~\cite[\eg][]{recht2019imagenet,DBLP:journals/ftml/KairouzMABBBBCC21}, in which several of the standard assumptions turn out to be limiting. Indeed, classical supervised learning assumes that, at training time, the examples  $(x_1,y_1), \dots, (x_n, y_n)$ are drawn i.i.d. from a given distribution $P$, and that, at testing time, we face a new example $x'$, also drawn from the same distribution $P$. However, recent failures of learning systems when operating in unknown environments~\cite{metz2018microsoft,knight2018selfdriving} underscore the importance of taking into account that we may not face the same distribution at test/prediction time.

Distributionally robust learning
aims to bolster the safety of learning systems
by enforcing robustness to heterogeneous data.
This notion of robustness is aligned with the one in robust optimization~\cite{DBLP:books/degruyter/Ben-TalGN09}; it is, however, different from the notion of robustness in robust statistics~\cite[Sec. 12.6]{DBLP:books/degruyter/Ben-TalGN09}. Here, we assume that the dataset is preprocessed to remove outliers such that the extreme data in the dataset is relevant to the learning process.

The superquantile can be used to build distributionally robust machine learning models, as studied recently in~\cite{laguel2020first,levy2020large,DBLP:conf/nips/CuriLJ020,soma2020statistical} among others.
From the dual formulation\;\eqref{eq:def_max_cvar}, superquantiles are expected to produce models that perform better in case of changes in underlying distributions, compared to models trained using standard empirical risk minimization.
Therefore, a natural approach to distributionally
robust learning consists in replacing the
expectation in\;\eqref{eq:pop-risk} by the superquantile\;\eqref{eq:def_cvar}. The resulting objective function is
\begin{equation}\label{eq:pop-sq}
\min_{\x \in W}
\left[
\Rp(w)=
{[\Sp]}_{(x,y)\sim P}\big[\ell(y, \varphi(\x,x)\big]
\right]\,,
\end{equation}
as well as its empirical version analogous to\;\eqref{eq:ERM}
\begin{equation}\label{eq:ESM}
\min_{\x \in W}
\left[
\Rpn(w)=
{[\Sp]}_{(x,y)\sim\Pn}\big[\ell(y, \varphi(\x,x )\big]
\right]\,.
\end{equation}
As we establish in the forthcoming Theorem\;\ref{thm:uniform-convergence-param},
the convergence property\;\eqref{eq:convergence}
also holds for $\Rpn$ and $\Rp$. We refer to \cite{DBLP:journals/corr/abs-1810-08750,shafieezadeh2019regularization} for further discussions on statistical aspects of
distributionally robust learning,
and to
\cite{mhammedi2020pac,soma2020statistical,lee2020learning} for superquantile learning in particular.

In practice, superquantile has been shown experimentally to produce  models more robust to distribution shifts
in various contexts; we refer to
\cite{DBLP:conf/nips/FanLYH17,laguel2020first,levy2020large,DBLP:conf/nips/CuriLJ020,soma2020statistical,DBLP:conf/aistats/KawaguchiL20}.
For illustration, we include numerical experiments inspired from the ones of\;\cite{DBLP:conf/nips/CuriLJ020} in the Appendix. We also include
a short toy example here.

\begin{example}[Superquantile regression]\label{ex:safe_learning}
We illustrate the interest of superquantile learning in presence of heterogeneous data, on a variant of the regression task of Example\;\ref{ex:learning}.
Consider a dataset gathering two different subgroups: $80\%$ of the points are generated by \eqref{eq:data} and the remaining $20\%$ are also generated by \eqref{eq:data} but with completely different parameters $\bar \x$.
Then we can compare the usual approach using ordinary least-squares\;\eqref{eq:OLS} with its superquantile counterpart of the form\;\eqref{eq:ESM} for $p=0.9$.

Figure\;\ref{fig:safe_plot} shows the distribution of residuals $r_i = |y_i - (\x_2 x_i^2 + \x_1 x_i + \x_0)|$ for models\;\eqref{eq:OLS} and\;\eqref{eq:ESM}.
The superquantile model\;\eqref{eq:ESM} shows an improvement of 90/95th quantiles of the distribution of residuals, which appears on histograms as a shift of the upper tail to the left.
This comes at the price of a degraded performance on average, which appears on the figure as the shift of the peak
of residuals to the right.\qed

\end{example}

\begin{figure}[t!]
    \centering
    \includegraphics[height=4cm]{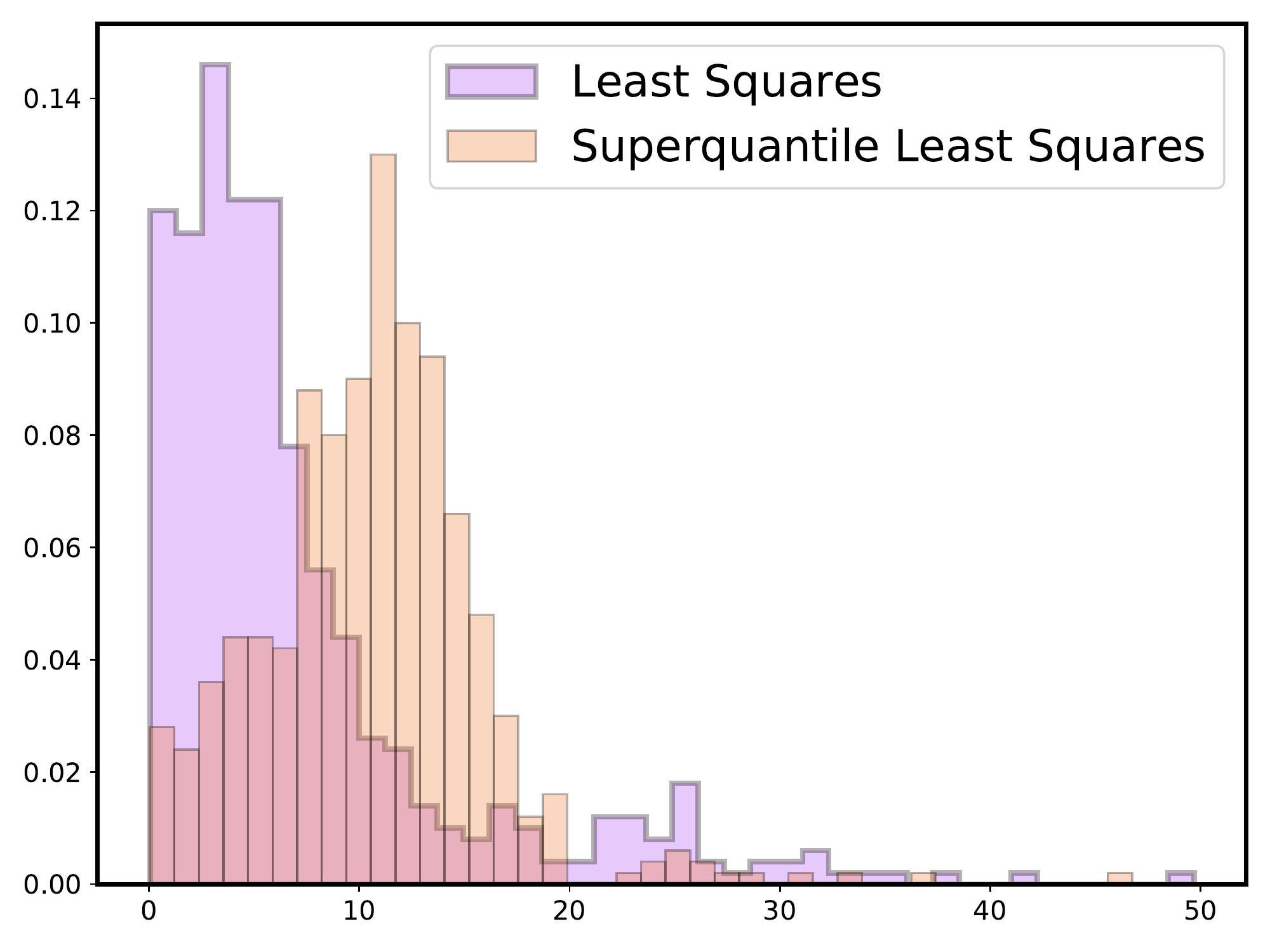}
    \hspace{2ex}
    \begin{tabular}[b]{lcc}
\toprule
Metric & model\;\eqref{eq:OLS} & model\;\eqref{eq:ESM} \\[2mm]
\midrule
Mean & $\mathbf{7.23}$ & $10.62$ \\[2mm]
$80^\text{th}$ perc. & $\mathbf{9.73}$ & $13.94$ \\[2mm]
$90^\text{th}$ perc. & $17.29$ & $\mathbf{15.93}$ \\[2mm]
$95^\text{th}$ perc. & $24.00$ & $\mathbf{17.83}$ \\[2mm]
\bottomrule
\vspace{1.5ex}
\end{tabular}
    \captionlistentry[table]{A table beside a figure}
    \caption{Superquantile regression improves the prediction on the worst-case datapoints. \textbf{Left figure}: histograms of residuals $r_i = |y_i - (\x_2 x_i^2 + \x_1 x_i + \x_0)|$ for model\;\eqref{eq:OLS} (in violet) and model\;\eqref{eq:ESM} (in orange).
    \textbf{Right table}: $x^\text{th}$ perc. stands for $x$\textsuperscript{th} percentile of final distribution of the residuals $r_i$.\label{fig:safe_plot}}
\end{figure}


We finish this section 
on superquantile learning with an asymptotic result generalizing~\eqref{eq:convergence} for the superquantile.
We present an elementary self-contained proof: we follow the general approach (see e.g.\;the monograph \cite{zbMATH05133436}) that we combine with the specific expression of the superquantile.


We start with the mathematical framework.
The input-output space $X \times Y$
is equipped with a $\sigma$-algebra $\mathcal{F}$,
which is complete with respect to the probability measure $P$.
The almost sure convergence is proved with respect to a countable sequence of datapoints $(x_1, y_1),\ldots$, the associated product $\sigma$-algebra and product probability measure, which we denote $\mathbb{P}$.
The prediction space $Z$ is equipped with a norm $\|\cdot\|$.
We consider the uniform pseudometric $\dist_\varphi$ on the parameter set $W$
\[
    \dist_\varphi(w, w') = \sup_{x \in X} \| \varphi(w, x) - \varphi(w', x) \|.
\]
We assume $W$ is bounded and we further make the following assumption on its size with respect to $\dist_\varphi$.

\begin{assumption}[On the "size" of $W$]\label{assump:cover}
For any $\varepsilon > 0$, there exists a finite set $T \subset W$
such that for every $w \in W$, there exists a
$w' \in T$ with $\dist_\varphi(w, w') \le \varepsilon$.
Such a $T$ is called a $\varepsilon$-cover of $W$,
and the size of the smallest such a $T$
is denoted $\covnum(\varepsilon)$.
\end{assumption}

For example for the set of $d$-dimensional linear functions
$\varphi(w, x) = \trans{w}x$ for $\|w\|_2 \le 1$. If we take
the norm $\|z\| = |z|$ on the real line $Z = \mathbb{R}$, one can prove that $\log\covnum(W, \dist_\varphi, \varepsilon) \le C\, d \log(1/\varepsilon)$
for some absolute constant $C$ and normalized data (see e.g., \cite[Lemma 5.7]{wainwright_2019}).
The second standard assumption that we consider is on the loss function $\ell(\cdot,\cdot)$.

\begin{assumption}[On the loss]\label{assump:loss}
The map $(x, y) \mapsto \ell(y, \varphi(w, x))$
is measurable for every $w \in W$. The map
$w \mapsto \ell(y, \varphi(w, x))$ is continuous, and hence Borel measurable, for each $(x, y) \in X \times Y$.
Furthermore, the loss is
\begin{itemize}
\item $P$-almost surely bounded, i.e., $0 \le \ell(y, \varphi(w, x)) \le B$ for each $w \in W$,
\medskip
\item $M$-Lipschitz in the second argument, 
i.e., $|\ell(y, z) - \ell(y, z')| \le M\|z - z'\|$  for every $z, z' \in Z$.
\end{itemize}
\end{assumption}

We have the following result generalizing \eqref{eq:convergence}.
\begin{thm}
\label{thm:uniform-convergence-param}
Let Assumptions \ref{assump:cover} and \ref{assump:loss} hold. Fix $p \in (0, 1)$ and
assume that the minimizers of $\Rp$ and $\Rpn$ 
are attained:
\[
w^\star \in \argmin_{w \in W} \Rp(w)
\quad\text{and}\quad w_n^\star \in \argmin_{w \in W} \Rpn(w).
\]
Then,
we have that $\Rp(w_n^\star) \to \Rp(w^\star)$ almost surely.
\end{thm}


\begin{proof}[Proof Sketch]
    We give a sketch here and defer technical details to Appendix.

    The key step in the proof is to show the uniform convergence
    \begin{equation}\label{eq:key0}
    \text{$\Rpn(w) \to \Rp(w)$ almost surely for all $w \in W$.}
    \end{equation}
    Indeed, once we have this,
    the result immediately follows as
	\begin{align}
		0 \le \Rp(w_n^\star) - \Rp(w^\star)
		&=
		\Rp(w_n^\star) - \Rpn(w_n^\star)
		+ \Rpn(w_n^\star) - \Rpn(w^\star)
		+ \Rpn(w^\star) - \Rp(w^\star)\\
		&\le
		2 \sup_{w \in W} |\Rpn(w) - \Rp(w)|
		\to 0,
	\end{align}
	where we use $\Rpn(w_n^\star) \le \Rpn(w^\star)$ in the second inequality.

    The proof of the uniform convergence follows
    the general approach (see e.g.\;\cite{zbMATH05133436}) adapted to variational expression of the superquantile\;\eqref{eq:def_min_cvar}.
    We introduce
	\[
		\bar \Rp(w, \eta)
		= \eta + \frac{1}{1-p} \mathbb{E}_{(x, y)\sim P}
		[\max(\ell(y, \varphi(w, x))-\eta,0)] \,,
	\]
	to write
	\[
	\Rp(w) = \min_{\eta \in [0, B]} \bar \Rp(w, \eta),
	\]
	as well as analogous empirical version $\bar \Rpn(w)$.
	The proof now consists in two steps, for a given $\varepsilon > 0$
    \begin{itemize}
        \item to construct a cover $T$\;of $W\!\times\![0,\!B]$ from a cover of\;$W$\,(given by assumption);
        \item to control the convergence over the points of $T$, more precisely to control the probability of the event
        \[
		E_n(\varepsilon) = \bigcap_{(w, \eta) \in T} \left\{ \bar \Rpn(w, \eta) - \bar \Rp(w, \eta) \le \varepsilon / 2 \right\} \,.
	    \]
    \end{itemize}
	In fact, we show that
	\begin{equation}
	    \sum_{n=1}^\infty
	    \mathbb{P}\Big(\sup_{w \in W }
	    \, |\Rpn(w) - \Rp(w)| > \varepsilon\Big)
	    \le
	    \sum_{n=1}^\infty \mathbb{P}\big(\overline E_n(\varepsilon)\big) < \infty,
	\end{equation}
from which we conclude the uniform convergence $\Rpn(w) \to \Rp(w)$ for all $w\in W$ using the Borel-Cantelli Lemma.
\end{proof}

\section{Recent Applications of the Superquantile in Machine Learning}\label{sec:newapplis}

We now turn to some recent applications of the superquantile in machine learning.


\subsection{Conformity in Distributed Learning on Mobile Devices}

The superquantile can be leveraged in distributed learning on mobile devices to model conformity to the population~\cite{laguel2020device}. Each mobile device contains the data generated by a single user, and thus the data distribution across devices is highly heterogeneous.

Concretely, suppose we have $\nfl$ training devices with respective data distribution $q_i$ and losses $\Loss_i(w) = \mathbb{E}_{(x, y) \sim q_i}[\ell(y, \varphi(\x, x))]$. {\em Federated learning}~\cite{DBLP:journals/ftml/KairouzMABBBBCC21}
is a distributed learning paradigm which
aims to collaboratively learn a common model across all devices without moving data between devices. The empirical risk minimization approach to federated learning consists in assigning a weight $\alpha_i>0$ to each device to minimize the aggregate loss, which corresponds to an expectation over a mixture $p_{\alpha}$ of the training distributions $q_i$:
\begin{equation} \label{eq:fl}
    \Loss(w) = \!\sum_{i=1}^{\nfl} \alpha_i \Loss_i(w)
    = \mathbb{E}_{(x,y) \sim p_\alpha}[\ell(y, \varphi(\x, x))] ~\text{ with}
    \left\{
    \begin{array}{ll}
        p_\alpha=\sum_{i=1}^{\nfl} \alpha_i q_i\,, \\
        \sum_{i=1}^n \alpha_i = 1\,.\\
    \end{array}
    \right.
\end{equation}
While minimizing such an objective might offer good performance for test devices which conform to the population of training devices (i.e., distribution $q$ of the test device is close to $p_\alpha$), one can expect poor predictive performance 
when $q$ largely departs from $p_\alpha$. An alternative is to model the heterogeneity of devices by considering data distribution $p_\pi$ written as a convex combinaison of the training distributions, but with weights $\pi_i$ different from\;$\alpha_i$:
\vspace*{-1ex}
\[
    p_\pi = \sum_{i=1}^{\nfl} \pi_i q_i, \quad\text{ with }
        0\leq \pi_i \leq 1 \text{ and }
        \sum_{i=1}^{\nfl} \pi_i = 1.
\]
In this context,~\cite{laguel2020device} proposes to measure how close a test device's distribution $p_\pi$ is to the training distribution $p_\alpha$ by the so-called conformity level
\begin{equation}\label{eq:ppi}
    \text{conf}(p_\pi) = \min_{1\leq i \leq {\nfl}} \alpha_i / \pi_i ~~ \in (0,1].
\end{equation}
We see that the closer the conformity level is to $1$, the closer $p_\pi$ is to $p_\alpha$, and thus the device and its user tightly conform to the population trend. To learn a robust model $\x$ performing well on reasonably non-conforming devices, \cite{laguel2020device} proposes to find the best $\x$ for the set of devices with a conformity of at least a given threshold $c\in (0,1)$; this leads to the optimization problem
\begin{equation}\label{eq:deltaFL}
    \min_{w \in \mathbb{R}^d} ~\max_{p_\pi \in \mathcal{P}} ~\mathbb{E}_{(x,y) \sim p_\pi}[\ell(y, \varphi(\x, x))]\,, \quad\text{with}\quad
    \mathcal{P} = \left\{p_\pi
    : \text{conf}(p_\pi) \geq c\right\}.
\end{equation}
We observe now that the condition $\text{conf}(p_\pi) \geq c$ can be written $\pi_i/ \alpha_i \leq 1/c$ for all $i$. Thus this constraint coincides, for the level $p = 1-c$, with the constraint $q_i \leq \frac{1}{{\nfl}(1-p)}$ in the dual formulation of the superquantile \eqref{eq:def_max_cvar}; see more precisely the discrete version of the dual formulation \eqref{eq:dual_discrete}. The extensive computational experiments of \cite[Sec.\;4]{laguel2020device} show that such superquantile federated learning has, as expected, superior performances for heterogeneous devices. Here we provide a toy example illustrating the interest of the approach.

\begin{figure}[t!]
\centering
\vspace*{-1ex}
  \centerline{\includegraphics[width=6.8cm]{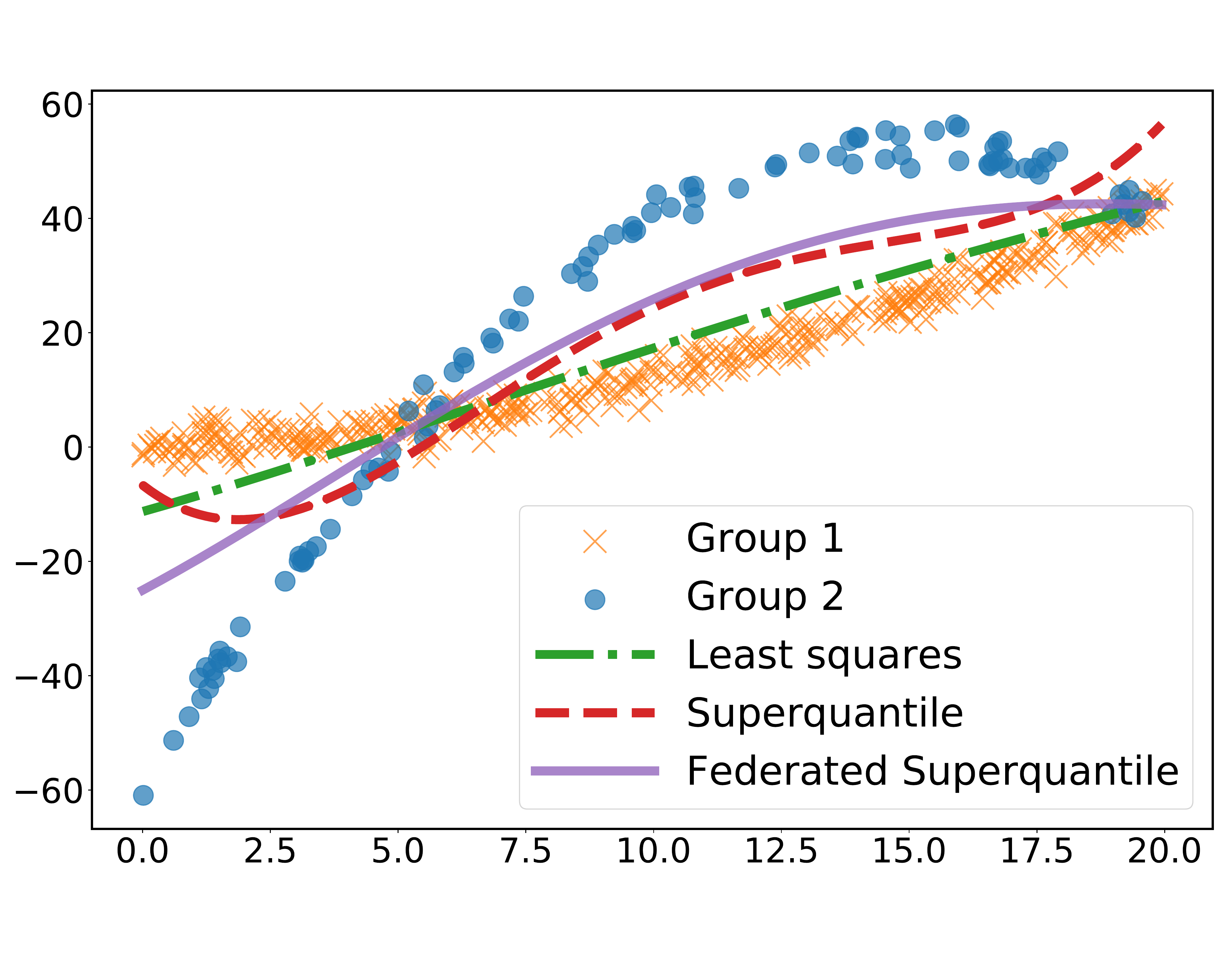}}
\vspace{-3ex}
\caption{Comparison of ERM~\eqref{eq:ERM}, superquantile minimization~\eqref{eq:ESM} and its federated counterpart~\eqref{eq:deltaFL} for the toy federated learning setting described in Example\;\ref{ex:fedreg}. 
We wish to have good performance on individual users.
Graphically, this is equivalent to fitting a curve to be at the same distance from the data-points of the conforming users (blue bullet) and the non-conforming user (orange cross).}
\vspace*{-2ex}
\label{fig:fairness_plot}
\end{figure}

\begin{example}[Federated regression]\label{ex:fedreg}
Consider a specific instance of Example~\ref{ex:safe_learning} in a federated setting. We consider that $80\%$ of the data
corresponds to four devices having the same data distribution following \eqref{eq:data}, while the remaining $20\%$ corresponds to a fifth device having its own distribution. In this example,
the solution to federated regression problem 
with $\alpha$ being the uniform distribution over the five devices coincides with the ordinary least squares model on the whole dataset. Figure\;\ref{fig:fairness_plot} shows this bivalent dataset: the blue points correspond to the data of the four first devices, and the red points correspond to the last device.  We would like to have a regression that captures worst-cases for both behaviours. We take the conformity level of $c = 1-p = 1/5$, using the knowledge on how the dataset is constructed.

Figure\;\ref{fig:fairness_plot} shows the the regression curves fit by \eqref{eq:ERM} (specifically, \eqref{eq:OLS} in this case), \eqref{eq:ESM}, and the \eqref{eq:deltaFL}. We can make three observations.
First, the standard model\;\eqref{eq:OLS} (in purple) tends to follow the trend imposed by the first four devices. Second, the superquantile model \eqref{eq:ESM} (in red) has better regression on worst-case data, irrespective of the group of the data point. Finally, the federated superquantile model\;\eqref{eq:deltaFL} finds, in contrast, a compromise between the two trends. Thus, federated superquantile regression better captures the (orange) data points of the non-conforming user.\qed
\end{example}

\vspace*{-1ex}
\subsection{Fairness-aware Machine Learning}

The superquantile naturally appears when considering the notion of fairness in machine learning (see \eg \cite{DBLP:conf/pkdd/KamishimaAAS12,DBLP:conf/chi/HolsteinVDDW19}), which we turn to next. 

Fairness in machine learning is studied with reference to a sensitive attribute, such as race or gender; see \eg \cite{DBLP:conf/pkdd/KamishimaAAS12}.
Suppose that we have a population that can be partitioned unambiguously between $\nfair$ subgroups with respect to this attribute. We denote $\Loss(\x) = (\Loss_1(\x), \ldots, \Loss_{\nfair}(w))$ the vector of losses on each of the subgroups of the training set.
Fairness in such situation would require independence between the sensible attribute and either predicted value or averaged losses per group $\Loss_i(\x)$.
An \emph{ideal group fairness} of the model $\x$ would then imply that $\Loss_1(\x) = \cdots = \Loss_n(\x)$~\cite[Def.\;1]{willaimson2019fairness}; but such a model could be no better than random guessing in the worst case. So \cite{willaimson2019fairness} considers \emph{approximate group fairness} and introduces the notion of fairness risk measures.
As explained in detail in Section 4 and supplementary material of\;\cite{willaimson2019fairness}, key properties for fairness risk measures includes convexity, positively homogeneity and monotonicity: 
the superquantile is thus a prominent example of such a measure. Experiments in  \cite[Sec.\;7]{willaimson2019fairness} show that superquantile indeed allows for a good balance between predictive accuracy and fairness violation. 
For completeness, we provide here a simple illustration in the context of Example\;\ref{ex:fedreg}.

\begin{example}[Fair regression]\label{ex:fairreg}
Let us come back to the toy example of Example\;\ref{ex:fedreg}. We look at it with the perspective of fairness between the predominant group (the four blue users) and the minority group (the fifth ``red'' user). Table~\ref{table:user_level_metrics} compares (i) the average performance over the predominant group and (ii) the average performance on the minority group. We observe that the difference between these performance is minimal for the user-level superquantile model provided by \eqref{eq:deltaFL}, achieving better approximate group fairness.\qed
\end{example}

\begin{table}[ht]
\vspace{-2ex}
\begin{center}
\begin{tabular}{lccc}
\toprule
Model  &  $\Loss_1(\x)$ (blue subgr.)  & $\Loss_2(\x)$ (red subgr.) \\
\midrule
ERM/least-squares \eqref{eq:OLS}& $4.59$ & $17.76$ \\
superquantile \eqref{eq:ESM} & $9.88$ & $13.62$ \\
federated superquantile \,\eqref{eq:deltaFL} & $10.87$ & $11.46$ \\
\bottomrule
\end{tabular}
\vspace{-1ex}
\caption{Average performances of each model over both subgroups.~\label{table:user_level_metrics}
\vspace{-2ex}}

\end{center}
\end{table}

\subsection{Adversarial Classification}

The superquantile also appears in generalized classification tasks when studying robustness to perturbations of data distributions\;\cite{DBLP:journals/corr/abs-2005-13815}.

In binary classification, we have $Y=\{-1,+1\}$ and the prediction function $\varphi(\x,x)$ correctly classifies a data point $(x,y)$ if
\[
y\,\varphi(\x,x) > 0.
\]
For an underlying data distribution $P$, we may want to choose $w$ so as either to minimize the probability  $\PP_{(x,y)\sim P}\big(y\,\varphi(\x,x) \leq 0\big)$ of encountering an error, or to control the distance $d(\x,(\bar x,\bar y))$ to misclassification of a data point $(\bar x,\bar y)$:
\[
d(\x,(\bar x,\bar y)) =
\inf_{x}
    \left\{
    \|x-\bar x\|^2\,:\,
    \bar y\,\varphi(\x,x) \leq 0
    \right\}\,.
\]
In this context,
robustness against perturbations of the data distribution $P$ can be guaranteed by
minimizing the worst-case error probability over a ball (\eg for Wasserstein distance $d_{\text{W}}$) around $P$ 
\begin{equation}\label{eq:wasser_class}
    \min_{\x} \sup_{Q:\; d_{\text{W}}(Q,P)\leq \epsilon}
\mathbb{P}_{(x,y)\sim Q}\big(y\,\varphi(\x,x) \leq 0\big) \,.
\end{equation}
Interestingly, optimal solutions of this problem coincide with those
solving:
\begin{equation}\label{eq:advclass}
\min_{\x}~ [\Sp]_{(\bar x,\bar y)\sim P}\big(-d(\x,(\bar x,\bar y))\big)     \,,
\end{equation}
for a well-chosen $p$; see \cite[Theorem\;2.6]{DBLP:journals/corr/abs-2005-13815}.
When the distance function $d$ has a computable closed form, formulation~\eqref{eq:advclass} is simpler to handle than~\eqref{eq:wasser_class}.
We refer to\;\cite{DBLP:journals/corr/abs-2005-13815} for results in the general case and for related literature.
%

\subsection{Risk-Sensitive Reinforcement Learning}
In a framework different from the supervised learning ones considered so far, the superquantile plays a role in risk-sensitive reinforcement learning. Reinforcement learning methods attempt to find decision rules to minimize a cumulative cost~\cite{sutton2018reinforcement} in a
sequential decision-making setting.

Concretely, a learning agent acts in a Markov decision process 
using a policy\;$\pi$ which maps a state
to a distribution over an action space.
The agent's aim is to minimize the total cost $c(\tau) = \sum_{i=1}^n c(s_i)$ of a trajectory $\tau=(s_1,\ldots,s_n)$ of states taken by the agent while following the policy $\pi$.
Letting $\Gamma(\pi)$ denote the induced distribution over trajectories of length $n$ under policy $\pi$, standard reinforcement learning methods minimize the expected cumulative cost as
\[
 \min_{\x} \mathbb{E}_{\tau \sim \Gamma(\pi_{\x})}[c(\tau)] \,,
\]
where $\x \in \RR^d$ parameterizes the policy $\pi_w$.
The so-called policy gradient methods aim to solve this by first-order optimization methods where the gradient of the objective is estimated by Monte Carlo simulations~\cite{sutton2000policy}.

However, in safety-critical applications, we are interested in accounting for unlikely events with high cost~\cite{zbMATH03378669}. In particular, \cite{morimura2010nonparametric} considers sensitivity to risky high-cost trajectories by minimizing the superquantile counterpart
\begin{equation}\label{eq:RL}
    \min_{\x} \, [\Sp]_{\tau \sim \Gamma(\pi_{\x})}[c(\tau)] \,.
\end{equation}
This risk-sensitive reinforcement learning setting thus leads to similar superquantile problems as in the supervised learning setting. We refer to\;\cite{tamar2015policy} on how to adapt policy gradient methods to estimate the gradient of the objective with respect to the parameters.

\section{Efficient (Sub)differential Calculus}\label{sec:sub}

The applications 
sketched in the previous section reveal optimization problems with objective functions\footnote{In coherence with the previous section and to comply with common notation in machine learning, we stick to the notation $\x$ for the variable of the functions.} written as the composition of a superquantile and a general loss function
\begin{equation}\label{eq:f}
f(\x) = \Sp(\Loss(\x)).
\end{equation}
For example,\;\eqref{eq:ESM} involves $L\colon\RR^d\rightarrow\RR^n$ defined component-wise for each data point by $L_i(w) =\ell(y_i, \varphi(w,x_i))$; similar expressions follow from \eqref{eq:deltaFL}, \eqref{eq:advclass} and \eqref{eq:RL}.
We notice first that $L$ is usually non-convex (\eg with $\varphi$ as\;\eqref{eq:NN}) but smooth (\eg with $\ell$ as\;\eqref{eq:losses}).

In this section, we provide easy-to-implement expressions of subgradients of superquantile-based functions~\eqref{eq:f}, in Section\;\ref{sec:sub2}, and of gradients of smoothed approximations of them in Section\;\ref{sec:smooth}. Finally, in Section\;\ref{sec:smooth-comp}, we compare the proposed smoothing with others considered in the literature (\eg\cite{DBLP:journals/coap/ChenM96,luna2016approximation}). Computing the (sub)gradients would be the first step toward using first-order optimization algorithms for solving superquantile problems. Though simple, this idea of using first-order methods is not widely used for such problems; among the few exceptions, we mention the PhD thesis\;\cite{miranda2014superquantile} using subgradient algorithms (in a special case) and our conference paper\;\cite{laguel2020first} presenting a toolbox for using first-order methods in superquantile learning. 
The developments of this section detail and extend those of \cite[Sec.~3]{laguel2020first}. 


\vspace{-1ex}
\subsection{Computing the Superquantile}\label{sec:comp}

\vspace{-0.5ex}
For the practical developments of this section, we consider a data-driven setting where the random variable $U$ takes equiprobable 
values $u_1,\ldots,u_n$\footnote{
In the sequel, we make a slight abuse of notation by not distinguishing between the random variable $U$ and the vector of its equiprobable realizations $u = (u_1,\ldots,u_n)$. Thus we consider the superquantile function $\Sp\colon\RR^n\to\RR$ as a function of $u\in \RR^n$, and we study the differentiability properties of compositions with $\Sp$.
}. 
In this setting, the three representations of superquantile
recalled in Section\;\ref{sec:setting} take explicit forms that are especially
interesting from a computational perspective.
\begin{figure}[!t]
\centering
  \centerline{\includegraphics[width=9cm]{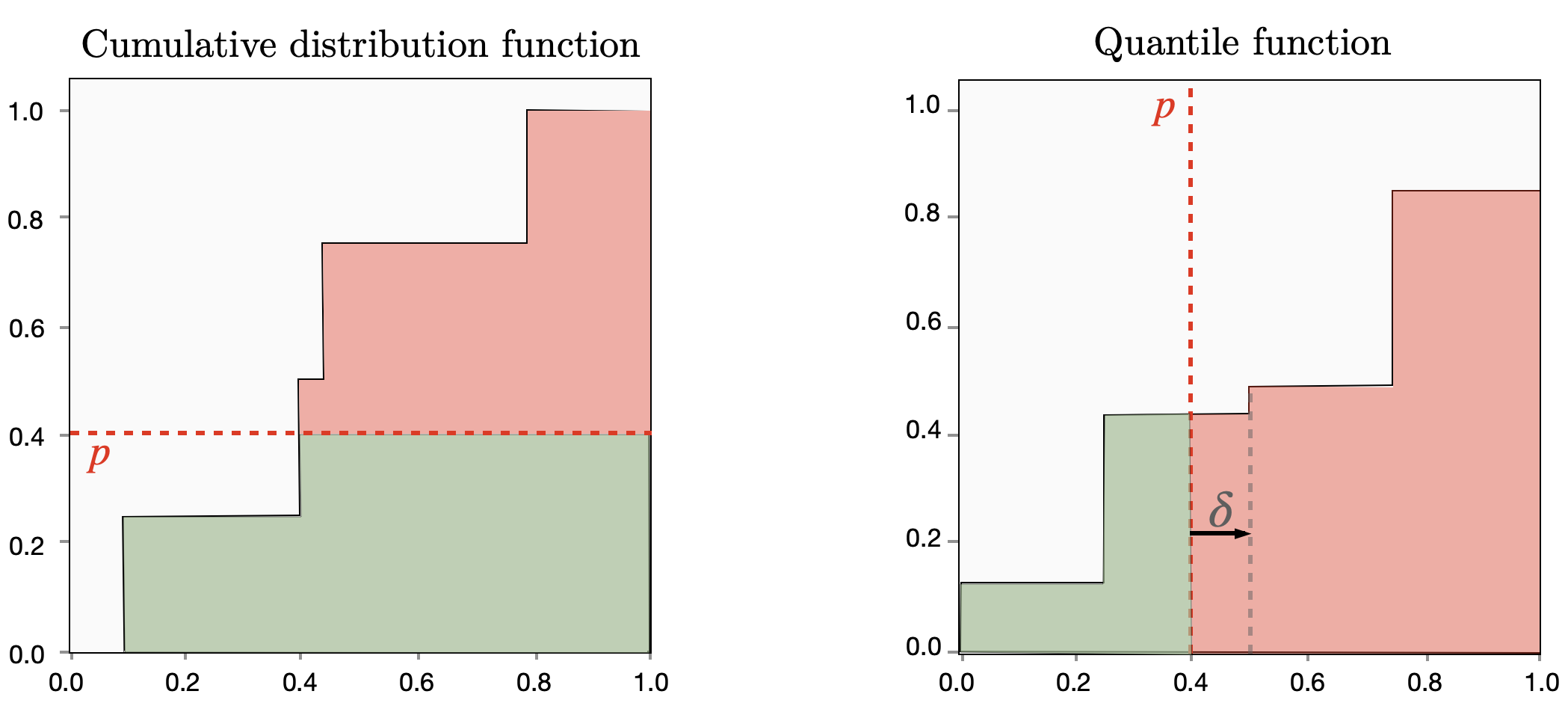}}
\caption{Illustration of the integral expression of the superquantile. Cumulative distribution function (on the left) and quantile function (on the right) are inverse one of the other. $\Sp(U)$ is obtained by averaging the quantiles greater than $p$ (red section on right graph).\label{fig:delta}}
\end{figure}

\begin{itemize}
    \item \textit{Integral representation.} By splitting the integral, \eqref{eq:def_cvar} can be written as
    \begin{equation}\label{eq:integral_discrete}
    \Sp(U) = \frac{1}{n(1\!-\!p)}\!\! \sum_{~i \in I_>} \!\!u_i  
        + \frac{\delta}{1\!-\!p} \Qp(U)
        ~~~\text{with $I_> \!= \{i: u_i\!>\!\Qp(U)\}$.}
        \end{equation}
    This expression involves the distance from $p$ to the next discontinuity point of the quantile function $p' \mapsto Q_{p'}(U)$ (see Figure\;\ref{fig:delta}):
    \begin{equation}\label{eq:delta}
    \delta = F_U(Q_p(U))-p = \frac{1}{n} (n - |I_>|) - p.
    \end{equation}
    Thus \eqref{eq:integral_discrete} gives an efficient way to compute superquantiles from the following three step procedure: (a) compute the $p$-quantile with the specialized algorithm (called \texttt{quickfind}) of complexity $\bigO(n)$;
(b) select all values greater or equal than the quantile;
(c) average values along \eqref{eq:integral_discrete}.

\medskip
\item \textit{Dual representation.} The expression \eqref{eq:def_max_cvar} simplifies to
\begin{equation}\label{eq:dual_discrete}
\Sp(U)= \max_{q \in \Delta_p} ~\trans{q}\!u
~~\;\text{with } \Delta_p\!=\!\left\{\!q \in \mathbb{R}^n_+\!:\! \sum_{i=1}^n q_i = 1,  q_i \le\!\frac{1}{n(1-p)}\!\right\}\!.
\end{equation}
In words, the superquantile is the support function of the intersection of the simplex with a box (see Figure\;\ref{fig:support}). This problem also corresponds to a classical optimization problem, called fractional knapsack problem, which is solved, after sorting the\;$u_i$'s, by a simple greedy strategy of the associated $q_i$'s~\cite{zbMATH07061852}.
For our purposes, this expression of
superquantile, as a direct max, is useful when applying dual smoothing techniques; see Section\;\ref{sec:smooth}.

\medskip
\item \textit{Variational representation.} The expression \eqref{eq:def_min_cvar} writes
\begin{equation}\label{eq:min_discrete}
    \Sp(U) = \inf_{\eta \in \mathbb{R}} \left\{\eta + \frac{1}{n(1-p)} \sum_{i=1}^n \max(u_i - \eta, 0)\right\}.
\end{equation}
This expression is often used in solving approaches for superquantile optimization; see \eg the progressive hedging for risk-averse multistage programming of \cite{rockafellar2018solving}.
Here, this expression will
provide a nice interpretation of the infimal convolution smoothing
(Corollary\;\ref{cor:corres}).
\end{itemize}

\begin{figure}[!t]
\centering
  \centerline{\includegraphics[width=4.5cm]{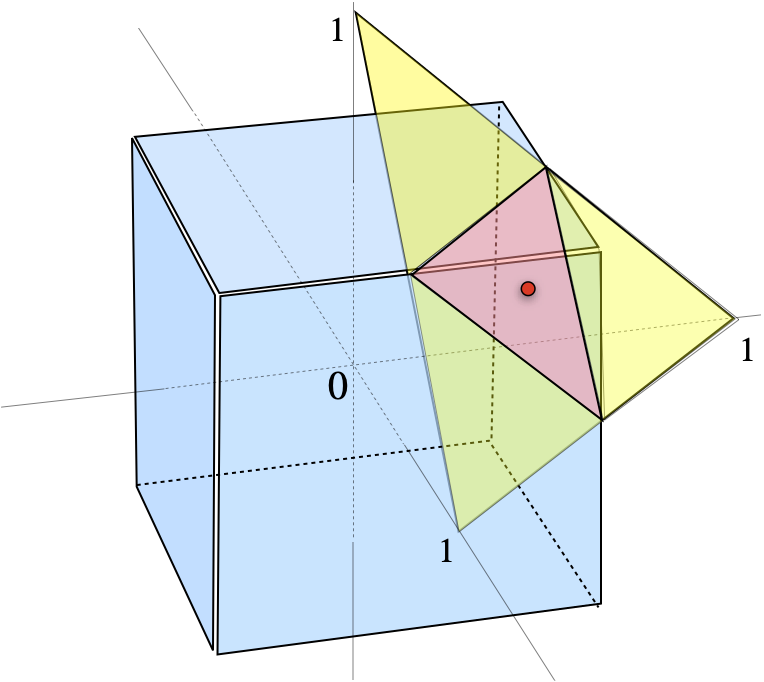}}
\caption{Illustration of the dual expression of the superquantile. $\Sp$ is the support function of the red
polytope. The red 
point represents the uniform distribution.
\vspace*{-1ex}\label{fig:support}}
\end{figure}


\vspace{-2ex}
\subsection{Subdifferentials}\label{sec:sub2}

In this section, we provide explicit and implementable expressions of the subdifferential of superquantile-based functions. Expressions of (convex) subdifferential of superquantile are well-known in general settings; see \eg\cite{ruszczynski2006optimization} for a thorough study. Here we study non-convex subdifferentials and derive concrete expressions in the data-driven context; we give direct proofs as applications of basic definitions and properties of nonsmooth analysis.


We start by recalling the standard notions of subgradients for nonsmooth functions (in finite dimension), following the terminology of\;\cite{rockafellar2009variational}.
For a function $\psi\colon\RR^d\rightarrow\RR\cup\{+\infty\}$,
the regular (or Fr{\'e}chet) subdifferential of $\psi$ at $\bar{\x}$ (such that $\psi(\bar \x) <+\infty$) is defined by
\[
\partial^R \psi(\bar{\x}) =\left\{s\in \mathbb{R}^d: ~\psi(\x)\geq
\psi(\bar{\x})+{s}\T(\x-\bar{\x}) + \smallo(\|{\x-\bar{\x}}\|)\right\}.
\]
The regular subdifferential thus corresponds to the set of gradients of smooth functions that are below $\psi$ and coincide with it at $\bar\x$.
The limiting subdifferential is the set of all limits produced by regular subgradients
\[
\partial^L \psi(\bar{\x})= \limsup_{\x\rightarrow \bar{\x}, \psi(\x)\rightarrow \psi(\bar{\x})}
\partial^R \psi(\x).
\]
These notions generalize (sub)gradients of both smooth functions and convex functions. The two subdifferentials coincide and reduce to the singleton $\{\nabla \psi(\bar \x)\}$ when $\psi$ is smooth and to the standard subdifferential from convex analysis when $\psi$ is convex.

For the function\;\eqref{eq:f}, which is the composition of a convex function and a continuously differentiable function, we get
from basic chain rules that the two subdifferentials coincide; we simply denote it by $\partial f(\x)$. Moreover the dual representation\;\eqref{eq:dual_discrete} expressing $\Sp$ as a support function allows to obtain readily an expression of the subdifferential of $\partial \Sp$ and, as a result, of the one of $f$. We formalize all this in the following proposition.


\begin{prop}[Explicit subdifferential of superquantile-based functions]\hfill\label{thm:sub}
Consider the superquantile-based function \eqref{eq:f} with $L$ continuously differentiable. We have
\begin{equation}\label{eq:sub}
    \partial f(\bar\x) = \Big( \partial^L f(\bar\x) =  \partial^R f(\bar \x) =\!\Big)~ \nabla\Loss(\bar \x)^*\partial\, \Sp(\Loss(\bar\x))
\end{equation}
where $\nabla\Loss(\bar\x)^*$ is the adjoint of the Jacobian of $\Loss$ at $\bar\x$ and $\partial\, \Sp(\Loss(\bar\x))$ the (convex) subdifferential of $\Sp$ taken at $L(\bar\x)$.
Moreover, for $\x \in \mathbb{R}^d$, compute $L(\w)\in\RR^n$ and $Q_p(L(\x))\in \RR$. Consider $I_>$ the set of indices such that $L_i(\x) > Q_p(L(\x))$ and $I_=$ the set of indices such that $L_i(\x) = Q_p(L(\x))$. Then the subdifferential of $f$ at $\x$ can be written with the gradients $\nabla L_i(\x)$ for $i\in I_>\cup I_=$, as follows
\begin{equation}
\partial f(\x)
~=~ \frac{1}{n(1-p)}\!\sum_{i\in I_>} \nabla L_i(\x)
~+~ \frac{\delta}{1-p} \conv\left\{ \nabla L_i(\x) : i \in I_=\right\}.
\end{equation}
\end{prop}

\begin{proof}
We apply the chain rule of \cite[10.6]{rockafellar2009variational} to the composition $\Sp\circ L$: we have that $\Sp$ is convex with full domain, which implies that the two subdifferentials\footnote{
Remark on the Clarke subdifferential:
As another by-product of the chain rule \cite[Thm. 10.6]{rockafellar2009variational}, the set of horizon subgradients of $f$ is reduced to $0$ since so is the one of $\Sp$ (convex and defined on $\RR^n$).}
As a consequence, the Clarke subdifferential is the convex hull of the limiting subdifferential \cite[Thm. 8.49]{rockafellar2009variational}. Thus we have, in our case, that the three subdifferentials (regular, limiting and Clarke) coincide.
of $f$ coincide (\ie $f$ is regular in the terminology of \cite{rockafellar2009variational}) and we have \eqref{eq:sub}.

Since $\Sp$ is the support function of the set $\Delta_p$, standard subdifferential calculus~\cite[Cor.\;4.4.4]{hull} gives that
$\partial \Sp(\Loss(\x))$ is the set of optimal solutions of~\eqref{eq:dual_discrete} with $\X=\Loss(\x)$. Knowing $I_>$ and $I_=$, the so-called fractional knapsack problem\;\eqref{eq:dual_discrete} can be solved by the simple greedy strategy~\cite{zbMATH07061852} of taking the largest $q_i$ for $i\in I_>$ and completing to $1$ with the $q_i$ for $i\in I_=$. Thus
\[
\text{$q$ 
solution of\;\eqref{eq:dual_discrete}} \iff \left\{
    \begin{array}{ll}
        q_i = \frac{1}{n(1-p)} & \mbox{if } i \in I_> \\
        0 \leq q_i \leq \frac{1}{n(1-p)} & \mbox{if } i \in I_= ~~\text{s.t.~} \sum_{i\in I_=} q_i = \frac{\delta}{1-p}\\
        q_i = 0 & \mbox{otherwise. }\\
    \end{array}
\right.
\]
By~\eqref{eq:sub}, this gives:
\[
    \partial f(w) = \frac{1}{n(1-p)} \sum_{i \in I_>} \nabla L_i(w) + \left\{\sum_{i\in I_=} q_i \nabla L_i(w), \text{ s.t. } \left\{
    \begin{array}{ll}
        0\leq q_i \;\forall i \in I_=\\
        \sum_{i \in I_=} q_i = \frac{\delta}{1-p} \\
    \end{array}
\right. \right\}.
\]
Finally, introducing weights $\alpha_i = \frac{q_i(1-p)}{\delta}$ for $i\in I_=$, the right-hand term can be written as the convex hull of $\nabla \Loss_i(\x)$ for $i\in I_=$, which gives the expression.
\end{proof}

We observe that the expression of $\partial f(\x)$ does not involve the gradients of all the $L_i$'s, but only of those associated to the largest values.
We also see that $f$ is differentiable at\;$\x$ if and only if $I_=$ is reduced to a singleton. The objective function is not differentiable in general, which poses a problem for a direct application of machine-learning gradient-based algorithms.

\subsection{Smoothing by Infimal Convolution}\label{sec:smooth}

In this section, we study a smoothing of nonsmooth superquantile-based functions \eqref{eq:f}. We propose to use the infimal convolution smoothing of~\cite{nesterov2005smooth}; the comparison to other smoothing approaches is postponed to the next section.

Following the guidelines laid out by \cite{beck2012smoothing}, we smooth only the superquantile $\Sp$ rather than the whole function\;$f$. Thus, we consider
\begin{equation}\label{eq:fnu}
\fnu(\x) = \Spnu(\Loss(\x))\,, \quad \text{for $\Spnu$ a smooth approximation of $\Sp$.}
\end{equation}
Regularizing the dual representation~\eqref{eq:dual_discrete} of superquantile, we consider the function, parameterized by the smoothing parameter $\nu$,
\begin{equation}\label{eq:primal}
\Spnu(\X) = \max_{q \in \Delta_p} \left\{ \trans{q}\X - \nu D(q) \right\}\,,
\end{equation}
for a given strongly convex function $D$. The following proposition establishes that the resulting function $\fnu$ as \eqref{eq:fnu} is a smooth approximation of $f$, as a direct application of \eg\citep[Theorem 4.1, Lemma 4.2]{beck2012smoothing}, or \cite[Theorem\;1]{nesterov2005smooth}.

\begin{prop}[Smoothed approximation]\label{thm:smooth_approx}
The function $\fnu$ defined by \eqref{eq:fnu} (with $\Spnu$ in  \eqref{eq:primal}) provides a global approximation of $f$, i.e.,
\[
\fnu(\x) \leq f(\x) \leq  \fnu(\x) + \frac{\nu}{2} \qquad \text{for all $\x \in \mathbb{R}^d$.}
\]
Moreover $\Spnu$ is differentiable, with $\nabla \Spnu(u)$ being the argmax of \eqref{eq:primal}, unique by strong convexity of $D$.
When $L$ is differentiable,\;$\fnu$\;is differentiable as well,\;with
\begin{equation}\label{eq:nablafnu}
\nabla \fnu(\x) = \jaco \Loss(\x)^*\nabla \Spnu(\Loss(\x)).
\end{equation}
\end{prop}


In our quest for simple and implementable expressions, we study in the rest of this section the case of separable strongly functions of the form:
\begin{equation}\label{eq:D}
D(q)= \sum_{i=1}^n d(q_i) \qquad \text{given a strongly convex function $d\colon [0,1]\rightarrow\RR$}.
\end{equation}
In Corollary\;\ref{cor:gradSnu}, we provide a general scheme to compute the gradient with explicit expressions in Examples\;\ref{ex:l2}\;and\;\ref{ex:KL} for special choices of $d$.
Finally we discuss the role of the smoothing parameter $\nu$ in a numerical illustration.

We start with a lemma gathering the nice duality properties of \eqref{eq:primal}.
A one-dimensional convex function plays a special role: it is
the convex conjugate of the sum of $\nu d$ and the indicator of the segment $\left[0, 1/n(1-p)\right]$
\begin{equation}\label{eq:gnu}
\gnu(s) = \left(\nu d + i_{\left[0, \frac{1}{n(1-p)}\right]}\right)^*(s)= \max_{0\leq t \leq \frac{1}{n(1-p)}} \left\{ s\, t -\nu \, d(t) \right\}\,.
\end{equation}
Since $d$ is strongly convex, standard (one-dimensional) convex analysis gives (see \eg\cite[Prop.~I.6.2.2]{hull}) that $\gnu$ is continuously differentiable with derivative
$\gnu'(s)$ being the (unique) $t$ achieving the above max. Simple calculus yields
\begin{equation}\label{eq:qnu}
\gnu'(s) 
= \left\{
    \begin{array}{cll}
        0 &&\mbox{ if } s\leq \nu\; d'_+(0)\\[1ex]
        \frac{1}{n(1-p)}  &&\mbox{ if } s \geq \nu\; d'_-(1/(n(1-p)))  \\[1ex]
        (d^*)'\left(\frac{s}{\nu}\right)&& \text{ otherwise.}\\
    \end{array}
    \right.
\end{equation}
where $d'_+(0)\in \left[ -\infty, +\infty \right)$ and $d'_-(1/(n(1-p)))\in \left[ -\infty, +\infty \right)$ are respectively the right-derivative of $d$ at $0$ and the left-derivative of $d$ at $1/(n(1-p))$. Note finally that ${\gnu}'$ is a non-decreasing function.

\begin{lem}[Duality]\label{lem:dual}
The dual problem of the convex problem~\eqref{eq:primal} (with a separable $D$ as in \eqref{eq:D}) can be expressed as the (smooth convex) one-dimensional problem:
\begin{equation}\label{eq:dual}
\min_{\eta} ~~ \theta(\eta) = \eta + \sum_{i=1}^n \gnui(\X_i-\eta).
\end{equation}
Moreover, there is no duality gap between \eqref{eq:primal} and \eqref{eq:dual}. There exists a primal-dual solution $(\qnu^\star, \etas)$ and the unique primal solution can be written $\qnu^\star = (\gnu'(u_i-\eta^{\star}))_{i=1,\ldots,n}$
with the help of \eqref{eq:qnu}.
\end{lem}

\begin{proof}
This lemma could be proved by applying a sequence of results from abstract Lagrangian duality\;\cite[Chap.\,XII]{hull}. Instead, we provide a simple proof from the direct calculus developed so far.
Consider the dualization of the constraint $\sum_{{i=1}}^n q_i - 1 = 0$
in\;$\Delta_p$.
For a primal variable $q\in B_p = \left[0, \frac{1}{n(1-p)}\right]^n$
and a dual variable $\eta\in \RR$, we write the Lagrangian
\[
L(q,\eta) = \sum_{{i=1}}^n q_i \X_i - \nu d_i(q_i) -  \eta \Big(\sum_{{i=1}}^n q_i-1\Big) = \eta + \sum_{{i=1}}^n q_i (\X_i-\eta) - \nu d_i(q_i) \,,
\]
and the associated dual function
\[
\theta(\eta) = \max_{q\in B_p} L(q,\eta)
= \eta + \sum_{{i=1}}^n \max_{0\leq q_i \leq \frac{1}{n(1-p)}} \left\{ q_i (\X_i-\eta) - \nu\, d_i(q_i) \right\}\,,
\]
which gives the expression of the dual function \eqref{eq:dual} from \eqref{eq:gnu}.
We have the so-called weak duality inequality by construction:
\begin{equation}\label{eq:weak}
\theta(\eta)  \geq L(q,\eta) = \sum_{{i=1}}^n q_i \X_i - \nu d_i(q_i)\,,
\quad\text{for all $\eta$ and all feasible $q\in \Delta_p$.}
\end{equation}

Now recall that $\gnu$ in \eqref{eq:gnu} is differentiable and so is the dual function with
\begin{equation}\label{eq:thetaprime}
\theta'(\eta) = 1 - \sum_{{i=1}}^n \gnu'(\X_i-\eta) \,.
\end{equation}
The above expression also shows that
\[
\lim_{\eta \rightarrow +\infty}\theta'(\eta) = 1\qquad\text{and}\qquad\lim_{\eta \rightarrow -\infty}\theta'(\eta) = 1 - \sum_{{i=1}}^n \frac{1}{n(1-p)}  = \frac{-p}{1-p}.
\]
By continuity of $\gnu'$ and $\theta'$, this implies that there exists $\etas$ such that $\theta'(\etas)=0$, i.e., there exists\;a dual solution $\etas$. 
On the primal side, the compactness of $B_p$ and strong convexity of $d$ gives existence and uniqueness of the primal solution, denoted $\qnu^\star$.
Next, we have a simple consequence of \eqref{eq:thetaprime}: the vector $(\gnu'(\X_i-\etas))_{i=1,\ldots,n}$, which lies in $B_p$ by construction, is primal feasible. From \eqref{eq:weak} and uniqueness of the primal solution, this implies that $\qnu^\star = (\gnu'(\X_i-\etas))_{i=1,\ldots,n}$ and that there is no duality gap.
\end{proof}

From Lemma\;\ref{lem:dual}, we get an almost explicit expressions of values and gradients of the smooth approximation $\fnu$.

\begin{cor}[Oracle for smooth approximation]\label{cor:gradSnu}
Consider $\fnu$ defined by\;\eqref{eq:fnu} with $L$ differentiable. With $\etas$ an optimal solution of\;\eqref{eq:dual} with $\X_i = \Loss_i(\x)$, 
\vspace{-2ex}
\begin{align}
\fnu(\x)&= \etas + \sum_{i=1}^n \gnui(\Loss_i(\x)-\etas) \,,\\
\nabla \fnu(\x)&= 
\sum_{i=1}^n \gnu'(\Loss_i(\x)-\etas)\,\nabla L_i(\x)\,
\end{align}
where $\gnu$ and $\gnu'$ are given by \eqref{eq:gnu} and \eqref{eq:qnu}.
\end{cor}

\begin{proof}
The no-gap result of Lemma\;\ref{lem:dual} gives that $\Spnu(\X)$ is equal to the optimal value of \eqref{eq:dual}. This gives directly the above expression of $\fnu(\x)= \Spnu(\Loss(\w))$ with $\etas$ an optimal solution of \eqref{eq:dual} with $\X_i = \Loss_i(\x)$. Regarding the expression of the gradient,
Proposition\;\ref{thm:smooth_approx} states that $\nabla \Spnu(\X)$ is the optimal solution of \eqref{eq:primal}, and Lemma\;\ref{lem:dual} expresses it as $(\gnu'(\X_i - \eta^{\star}))_{i=1,\ldots,n}$. We then get the expression of $\nabla \fnu(\x)$ from \eqref{eq:nablafnu}.
 \end{proof}

Thus the computation of the first-order oracle of 
$\fnu$ boils down to solving the one-dimensional convex problem \eqref{eq:dual} with $\X_i = L_i(\x)$.
This easy task can be solved in general by bisection or higher-order schemes. Here Lemma\;\ref{lem:dual} allows us to make an additional simplification with an initial interval tightening. We can indeed shrink the segment where to find $\etas$ to two consecutive points in
\begin{equation}\label{eq:N}
    N = \left\{ u_i - \nu\, d'_+(0) , u_i - \nu\,d'_-\Big(\frac{1}{(n(1-p)}\Big)~~\text{$i=1,\ldots,n$}\right\}
\end{equation}
which is a set of special points regarding the structure of the dual function (recall\;\eqref{eq:qnu} and\;\eqref{eq:dual}).
Denoting $\etaa$ and $\etab$, defined respectively as the largest point in $N$ such that $\theta'(\etaa)\leq0$ and the smallest point in $N$ such that $\theta'(\etab)\geq0$, we get\;$\eta^\star$ by testing three cases:
\begin{itemize}
    \item if $\theta'(\etaa) = 0$, take $\etas = \etaa$\;; if $\theta'(\etab) = 0$, take $\etas = \etab$\;;
    \item otherwise, compute $\etas$ in the small interval $[\etaa,\etab]$.
\end{itemize}
The initial interval tightening thus boils down to having sorted points in the set\;$N$, which is obtained directly from sorting the given data.

Finally we emphasize that we can sometimes go one step further ahead and obtain explicit expressions of $\etas$ and thus, readily implementable expressions of $\nabla \fnu(\x)$. In the next two examples, we illustrate this for two cases of interest,
when we smooth the superquantile by a divergence to the uniform probability (which is at the center of $\Delta_p$; recall Figure\;\ref{fig:support}). In particular the smoothing detailed in the forthcoming Example\;\ref{ex:l2} was used in the numerical illustrations of Examples\;\ref{ex:safe_learning},\;\ref{ex:fedreg},
and\;\ref{ex:fairreg} (where the resulting smoothed superquantile optimization problems were solved by L-BFGS).


\begin{example}[Euclidean smoothing]\label{ex:l2}
We suggest to smooth the superquantile with the Euclidean distance to the uniform distribution
\begin{equation}\label{eq:Deuclidian}
D(q) = \frac{1}{2} \|q - \barq\|^2
\quad\text{with}\quad
\barq = \left(\frac{1}{n},\ldots,\frac{1}{n}\right).
\end{equation}
This corresponds to \eqref{eq:D} with
\[
d(t) = \frac{1}{2}\left(t - \frac{1}{n}\right)^2.
\]
In this case, elementary calculus gives
\[
d_-'(0)=-\frac{1}{n},\quad
d_{+}'\left(\frac{1}{n(1-p)}\right)=\frac{p}{n(1-p)},\quad\text{and}\quad
(d^*)'\left(\frac{t}{\nu}\right) = \frac{t}{\nu} + \frac{1}{n}
\]
so that we get from \eqref{eq:qnu} the following expression
\begin{equation}\label{eq:sol_q_i}
\begin{split}
  q_i= \gnu'(\X_i-\eta) &=  \left\{
    \begin{array}{cll}
        0 &&\mbox{ if } \eta \geq \X_i+\frac{\nu}{n} \\
        \frac{1}{n(1-p)}  &&\mbox{ if } \eta \leq \X_i - \frac{\nu}{n}\frac{p}{1-p}  \\
        \frac{\X_i - \eta}{\nu} + \frac{1}{n}&& \text{ otherwise,}\\ 
    \end{array}
    \right.
\end{split}
\end{equation}
for the entries of the solution of \eqref{eq:primal}.
We also have that $\theta'$ is piecewise linear in this case and that
\[
N = \left\{ x_i + \frac{\nu}{n} , x_i - \frac{\nu}{n}\frac{p}{1-p}~~\text{$i=1,\ldots,n$}\right\}.
\]
Therefore from $\etaa$ and $\etab$ in $N$, finding $\etas$ in the interval $[\etaa,\etab]$ simply reduces to interpolating linearly as
\[
\etas = \etaa - \frac{\theta'(\etaa)(\etab-\etaa)}{\theta'(\etab) - \theta'(\etaa)}.
\]
We can apply Corollary\;\ref{cor:gradSnu} to get an efficiently implemented expression of the gradient. Note that the obtained expression of $\nabla \fnu(\x)$ involves only the gradients $\nabla \Loss_i(\x)$ for largest values of $\Loss_i(\x)$ (comparable to the expression of $\partial \Loss(\x)$ in Proposition\;\ref{thm:sub}).
\qed
\end{example}

\begin{example}[KL smoothing]\label{ex:KL}
We use here the Kullback-Leibler divergence to the uniform probability
\[
d(q) = \sum_{i=1}^n q_i \log(q_i/\barq_i)
\quad\text{with}\quad
\barq = \left(\frac{1}{n},\ldots,\frac{1}{n}\right).
\]
which consists in taking $d(t) = t\log(t)$ in \eqref{eq:D}.
Elementary calculus then gives
\[
d_+'(0)=-\infty,\quad
d_{-}'\!\!\left(\frac{1}{n(1-p)}\right)= 1 - \log(n(1-p)),
~~\text{and}~~
(d^*)'\!\!\left(\frac{t}{\nu}\right)\! = \exp\!\left(\frac{t}{\nu}-1\right)
\]
which in turn yields
\[
\begin{split}
  \gnu'(\X_i-\eta) &=  \left\{
    \begin{array}{cll}
        \frac{1}{n(1-p)} &&\mbox{ if } \eta \leq \X_i+  \nu \left(\log(n(1-p)) -1\right) \\
        \exp{(\frac{\X_i - \eta}{\nu} - 1 )}&& \text{ otherwise}\\ 
    \end{array}
    \right. \\
    N &= \left\{ \X_i+ \nu\left(\log(n(1-p)) -1\right)~~\text{$i=1,\ldots,n$}\right\}.
\end{split}
\]
On the interval $[\etaa,\etab]$, we have that
\[
\theta'(\eta) = 1 - \sum_{i\in I} \frac{1}{n(1-p)} - \sum_{i\notin I}\exp{\left(\frac{\X_i - \eta}{\nu} - 1 \right)}
\]
with $I= \{i , \X_i+  \nu \left(\log(n(1-p)) -1\right) \leq \etaa\}$ the set of indices of points in $N$ smaller than $\etaa$. This yields
\[
\eta^\star = \nu \log\left(\frac{\sum_{i \notin I} \exp(\X_i/\nu - 1)}{1 - {|I|}/\big({n(1-p)}\big)} \right).
\]
We can then apply Corollary\;\ref{cor:gradSnu} to get the smoothed gradient.\qed
\end{example}

\begin{figure}[t!]
\centering
\vspace*{-1ex}
  \centerline{\includegraphics[width=14cm]{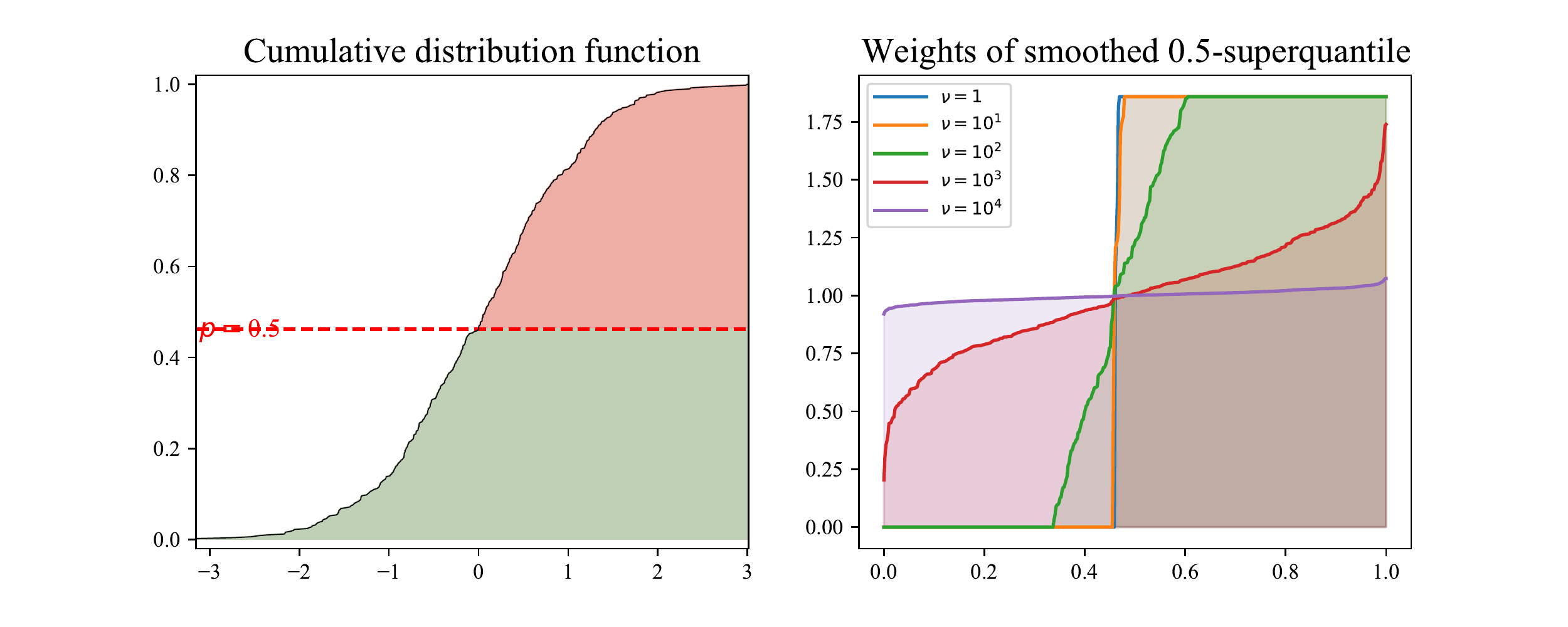}}
\vspace{-1ex}
\caption{Impact of the smoothing parameter $\nu$ on the weights assigned to the data points. \textbf{Left:} empirical cumulative distribution of $n=500$ points sampled from a standard Gaussian distribution. \textbf{Right:} distribution of weights, \ie the optimal solution of~\eqref{eq:primal} for $p=0.5$, with respect to sorted data points (\ie value at abscissa $t$ is the weight attached to the $t$-quantile).  Different colours correspond to
different values of $\nu$.}
\vspace*{-1ex}
\label{fig:sensit}
\end{figure}

We conclude this section on the infimal-smoothing of the superquantile with two remarks illustrating the impact of the smoothing parameter $\nu$.

\vspace*{-2ex}

\begin{remark}[Impact of the smoothing parameter on the weights]
We illustrate the impact of the smoothing parameter $\nu$ on the relative weights given to the data. We consider the Euclidean smoothing of Example\;\ref{ex:l2} with $p=0.5$. We sample $n=500$ points from a Gaussian distribution and compute the distribution of weights $q_i$ of \eqref{eq:sol_q_i}, solutions to smoothed problem~\eqref{eq:primal} for different values of the smoohting parameter $\nu$.
The right-hand side of Figure\;\ref{fig:sensit} displays the impact of $\nu$ of the obtained weights. In particular, we note that as $\nu$ grows, the distribution $q_i$ tends to spread uniformly over all data-points, so that the smoothed superquantile acts like the expectation. In contrast, when $\nu$ is close to 0, the distribution approximates the uniform distribution over the interval $[p,1]$,
 so that the smoothed superquantile acts like the superquantile. This approximation is further discussed in the next remark.\qed
\end{remark}

\begin{remark}[Impact of the smoothing parameter on the approximation]
We illustrate here the impact of the smoothing parameter $\nu$ on the approximation of the superquantile by its smoothed variant (Proposition\;\ref{thm:smooth_approx}). To do so, we fix a vector\;$\bar \x$ and we observe the values of smoothed approximations of a superquantile-based function for
different values of $\nu$.
More precisely, we consider the logistic regression problem used in Appendix\;\ref{sec:a:Numerical Illustration};
we use the quadratic smoothing of Example\;\ref{ex:l2} with $\nu=0.1$;
and we solve the problem by L-BFGS to get the reference point $\bar \x$.
We compute the following values at $\bar\x$:
\begin{itemize}
    \item the underlying superquantile-based objective \eqref{eq:ESM} which corresponds to the case $\nu=0$;
    \smallskip
    \item the smoothed approximations (which corresponds to \eqref{eq:ESM} with $\Spnu$ replacing $\Sp$) for a sequence of $\nu$ evenly spread on a log scale;
    \smallskip
    \item the usual empirical risk minimization objective \eqref{eq:ERM}, which corresponds to the case $\nu=+\infty$. Indeed, in this regime $\nu\to+\infty$, the impact of the quadratic penalization term $(q - \bar q)$ 
    increases 
    so that the solution of \eqref{eq:primal} eventually becomes the uniform distribution $\bar q$, in which case $\Spnu$ coincides with the expectation.
\end{itemize}

The observations from Figure\;\ref{fig:nu} are as expected. For small values of $\nu$, the difference between the superquantile-based objective and its smooth approximations vanishes. On the other hand, for large values of $\nu$, the smoothed superquantile loss tends to the average loss and does not approximate the nonsmooth superquantile loss well. 

A key benefit of smoothing the superquantile is to leverage efficient smooth optimization algorithms, such as L-BFGS, for superquantile learning.
When $\nu$ is too small, the problem is almost non-smooth, which leads to numerical issues with convergence (on this instance, L-BGFS fails to converge when $\nu$ too small or when used with the nonsmooth oracle of Proposition\;\ref{thm:sub} due to a line search failure).
When $\nu$ is too large, the smoothed superquantile gets close to the
expectation and the interest of using a superquantile approach disappears.
This 
illustrates the interest of having a moderate $\nu$ for superquantile learning,
where the smoothed objective is an reasonable approximation of the nonsmooth superquantile, while still being
smooth enough
to leverage fast optimization algorithms. \qed
\end{remark}

\begin{figure}[ht!]
\centering
\vspace*{-1ex}
  \centerline{\includegraphics[width=8cm]{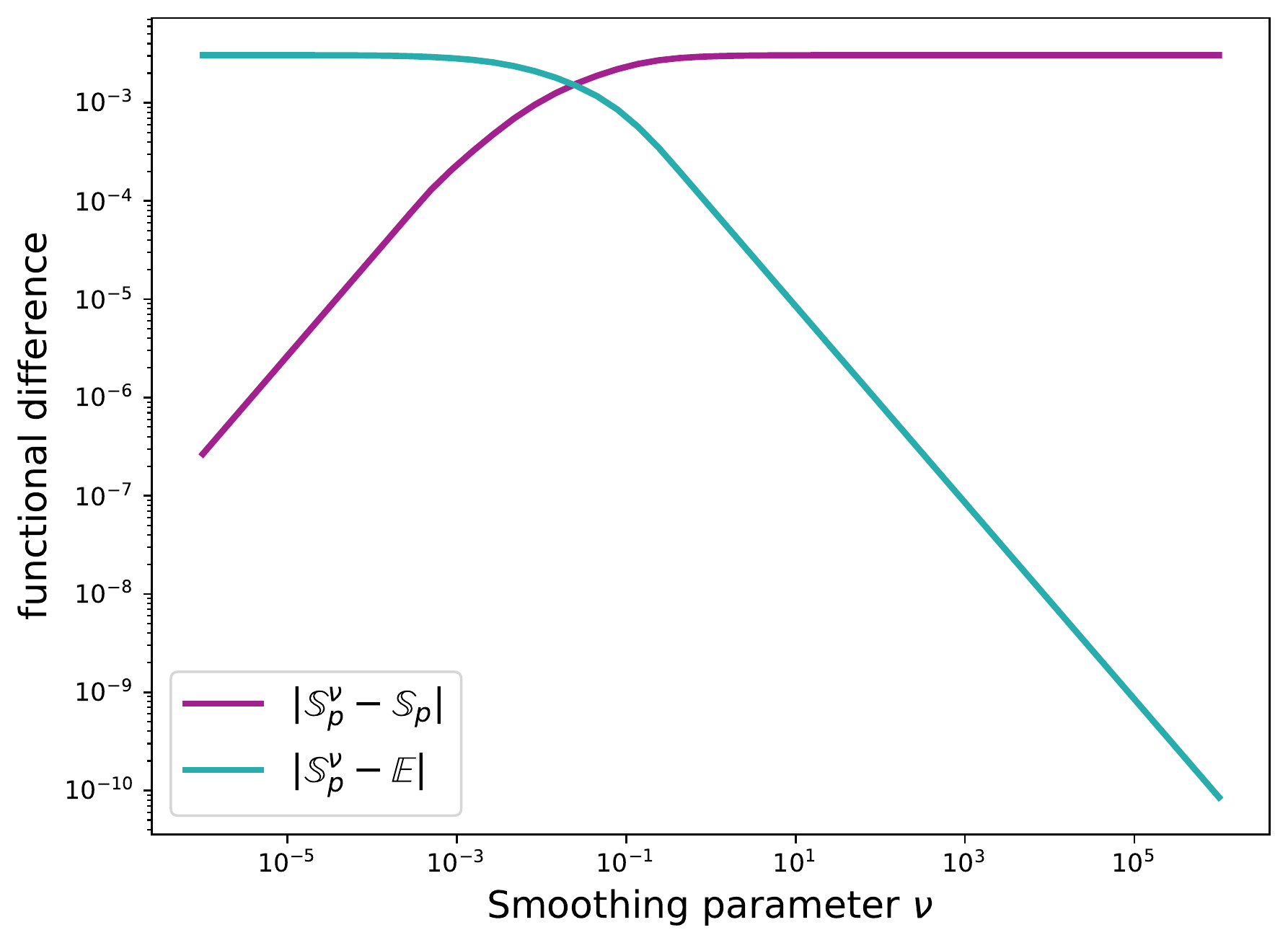}}
\vspace{-1ex}
\caption{Impact of the smoothing parameter $\nu$ when
solving a superquantile logistic regression on a classical dataset (Australian Credit dataset).}
\vspace*{-2ex}
\label{fig:nu}
\end{figure}

\vspace*{-1ex}
\subsection{Comparison to Other Smoothing Schemes}
\label{sec:smooth-comp}

We compare the proposed infimal convolution smoothing of the superquantile \eqref{eq:primal} to other possible smoothing schemes. Classical smoothing techniques are based either on convolution or infimal convolution. For superquantile, one could either smooth the dual representation\;\eqref{eq:dual_discrete} or the variational representation\;\eqref{eq:min_discrete}. Together, this yields four natural ways to smooth the superquantile.

We first formalize the equivalence between the two infimal convolution smoothings: indeed, smoothing the dual representation considered in the preceding section corresponds to a smoothing of $\max\{\cdot,0\}$ in the variational formulation.

\begin{cor}[Equivalence of smoothings with infimal convolution]\label{cor:corres}
With the notation of Section~\ref{sec:smooth}, the infimal convolution smoothing of $\Sp$ with a separable strongly convex function\;\eqref{eq:D} is equivalent to the infimal convolution smoothing of the positive part $\max\{\cdot,0\}$\;as
\begin{equation} \label{eq:smooth-iconv}
\mnui (\eta) = \!\max_{0\leq t \leq 1} \left\{\eta\,t -\nu \, \tilde d(t) \right\} \quad \text{ with ~$\tilde d(t)= n(1-p) d\left(\frac{t}{n(1-p)}\right)$}.
\end{equation}
More precisely, we have the following equality (to be compared with \eqref{eq:min_discrete})
\[
\Spnu(\X) = \min_{\eta} ~ \left\{\eta + \frac{1}{n(1-p)}\sum_{i=1}^n \mnui(\X_i-\eta)\right\}.
\]
\end{cor}

\begin{proof}
A direct change of variable in \eqref{eq:gnu} gives $\gnui(\X_i-\eta) =  \frac{1}{n(1-p)} \mnui(\X_i-\eta)$. The proof is direct from the expression of the dual problem \eqref{eq:dual} and the no-gap result stated in Lemma\;\ref{lem:dual}.
\end{proof}

Next, we show an equivalence between the
smoothing 
by infimal convolution\;\eqref{eq:smooth-iconv}, and by convolution, as considered in \cite{DBLP:journals/coap/ChenM96,luna2016approximation}.
Suppose we are given a continuous probability density $\rho\colon\RR\to\RR_+$
such that $\int_{-\infty}^{\infty} |s|\rho(s)\dd s$ is finite).
The smoothing by convolution of the function $\max\{\cdot, 0\}$ with density $\rho$ and
smoothing parameter $\nu > 0$ is
defined by\footnote{Applied to $\max\{x, \cdot\}$, the general smoothing by convolution as defined in~\eqref{eq:smooth-conv} coincides with the double integral representation used in\;\cite{DBLP:journals/coap/ChenM96,luna2016approximation}.
Indeed, integrating\;\eqref{eq:d-smooth-conv} yields
\[
\bar m_\nu(\eta) = \tfrac{1}{\nu} \int_{-\infty}^\eta
    \int_{-\infty}^{\eta'}  \rho\left(\tfrac{s}{\nu}\right) \dd s \, \dd \eta' \,.
\]
}
\begin{equation} \label{eq:smooth-conv}
    \bar m_\nu(\eta)
    = \frac{1}{\nu} \int_{-\infty}^\infty \max\{\eta-s, 0\} \rho\left(\tfrac{s}{\nu}\right) \dd s
    = \frac{1}{\nu} \int_{-\infty}^\eta (\eta - s) \rho\left(\tfrac{s}{\nu}\right) \dd s
    \,.
\end{equation}
The function $\bar m_\nu$ is convex and smooth, with derivative
\begin{equation} \label{eq:d-smooth-conv}
    \bar m_\nu'(\eta) = \frac{1}{\nu} \int_{-\infty}^\eta
    \rho\left(\tfrac{s}{\nu}\right) \dd s\,.
\end{equation}
The next proposition, relating this smoothing to the previous one, involves $Q_t(\rho)$ the quantile function of a random variable with density $\rho$.

\begin{prop}[Equivalence of convolution/inf-convolution smoothings]\hfill \label{prop:smooth-comp}
    With the above notation, the convolution smoothing $\bar m_\nu$ of~\eqref{eq:smooth-conv}\;for $\nu = 1$
    can be written as the infimal-convolution smoothing 
    (to be compared with\;\eqref{eq:smooth-iconv})
    \begin{equation} \label{eq:conv:1}
        \bar m_1(\eta) = \max_{0 \le t \le 1} \left\{ \eta\, t - \bar d(t)  \right\}
        \quad \text{where} \quad
        \bar d(t) = t Q_t(\rho) - \bar m_1(Q_t(\rho)).
    \end{equation}
    Conversely, the infimal convolution smoothing $m_\nu$ of~\eqref{eq:smooth-iconv}\;for\;$\nu = 1$
    can be written as the convolution smoothing 
    (to be compared with \eqref{eq:smooth-conv})
    \begin{equation} \label{eq:conv:2}
        m_1(\eta) = \lim_{s \to -\infty } m_1(s)  + \int_{-\infty}^\eta (\eta-s) \tilde \rho(s) \dd s
        \quad\text{where}\quad \tilde\rho(s)= m_1''(s) \text{ a.e}.
    \end{equation}
\end{prop}
\begin{proof}
    For the first part, we consider the convex conjugate of $\bar m_1$
    \[
        \bar m_1^*(t) = \sup_{\eta \in \RR} \left\{ \eta\, t - \bar m_1(\eta)\right\} \,.
    \]
    If $t \notin [0, 1]$, the supremum is $+\infty$ since $|\bar m_1(\eta) - \max\{\eta, 0\}|$ is bounded by an absolute constant. For $t \in [0, 1]$,
    the concave function $\eta \mapsto \eta t - \bar m_1(\eta)$ is maximized at
    $\eta^\star$ if and only if it satisfies the first-order optimality condition
    \[
    t= \bar m_1'(\eta^\star)=  \int_{-\infty}^{\eta^\star} \rho(s)\dd s.
    \]
    Since the latter is the cumulative distribution function,
    $\eta^\star = Q_t(\rho)$ is the corresponding quantile function (well-defined since $\rho$ is continuous). This yields
    \begin{equation}\label{eq:m1conj}
    \bar m_1^* = \bar d + i_{[0, 1]},
    \end{equation}
    which in turn gives\;\eqref{eq:conv:1}. Finally to establish the strong convexity of $\bar d$, we use again\;\eqref{eq:m1conj} together with the smoothness of $\bar m_1$. Thus $\bar m_1$ corresponds to the infimal-convolution smoothing with $\bar d$.

    For the second part, we start by noting that since $m'_1$ is Lipschitz, $m''_1$ exists almost everywhere, and 
    $\tilde \rho$ is well-defined.
    Since $m_1$ is convex, it also holds that $m_1''(s) \ge 0$, and then that we have the normalization
    \[
        \int_{-\infty}^\infty \tilde \rho(s) \dd s =
        \int_{-\infty}^\infty \tilde m_1''(s) \dd s =
        \lim_{\eta \to \infty} m_1'(\eta) - \lim_{\eta \to -\infty} m_1'(\eta) = 1 - 0 = 1 \,,
    \]
    where we use $m'(\eta)$ is the (unique) optimal solution of\;\eqref{eq:smooth-iconv}.
    Then the proof follows from the next two claims.

    \textit{Claim 1: $m_1$ admits a limit at $-\infty$.}
    Convexity of $m_1$ gives that $m_1'$ is non-decreasing. Since $\lim_{s \to -\infty} m_1'(s) = 0$, we get that $m_1'$ is non-negative. Thus, $m_1$ is non-decreasing and, since it is bounded from below, this implies that $m_1$ admits a limit at $-\infty$. We denote it as $m_1(-\infty)$.


    \textit{Claim 2: $\lim_{s \to -\infty} s\,m_1'(s) = 0$.}  For a given $s$, we write:
    \[
        s \; m_1'(2s) \leq  \int_{2s}^s m_1'(t) \dd t= m_1(s) - m_1(2s),
    \]
    where the inequality comes from the fact that $m_1'$ is non-decreasing. Using that $m_1$ admits a limit  at $-\infty$ (Claim 1), we get Claim 2.

    Finally, we can conclude the proof with integration by parts:
    \[
    \begin{split}
        m_1(\eta) &= m_1(-\infty) + \int_{-\infty}^\eta  m_1'(s) \dd s \\
        &= m_1(-\infty)  +  [(s-\eta) m_1'(s)]_{-\infty}^\eta + \int_{-\infty}^\eta (\eta  - s) \tilde \rho(s) \dd s \\
        &= m_1(-\infty) + \int_{-\infty}^\eta (\eta - s) \tilde \rho(s) \dd s.
    \end{split}
    \]
    This establishes \eqref{eq:conv:2} and ends the proof.
\end{proof}


Finally, we comment the smoothing of the dual representation \eqref{eq:dual_discrete} using convolution, which is defined as
\[
    \bar S_p^\nu(u) = \frac{1}{\nu}\int_{\RR^n} S_p(u-z) \rho\big(\tfrac{z}{\nu}\big) \, \dd z = \mathbb{E}_{Z\sim \rho} [S_p(u - \nu Z)] \,,
\]
for the density $\rho\colon\RR^n\rightarrow\RR$
and the parameter $\nu > 0$.
We do not consider this smoothing approach because it suffers from two drawbacks in view of practical implementation. First, it usually cannot be computed in closed form, unlike the other smoothing approaches considered here.
Second, the 
Lipschitz constant of the gradient (appearing in condition numbers, constant scalings, and rates of convergence of first-order methods\;\cite{nesterov-book}) scales badly:
as $O(\sqrt{n}/\nu)$ for the Lipschitz constant of $\nabla \bar S_p^\nu$~\cite[Lemma 2]{nesterov2017random}, as opposed to the dimension-independent $O(1/\nu)$ for the one of $\nabla S_p^\nu$~\cite[Theorem 1]{nesterov2005smooth}. 
\section{Conclusion}\label{sec:ccl}

In this paper, we have developed two different aspects of the superquantile, a famous risk measure studied and popularized by R.\,T.\,Rockafellar and co-authors. First, we have reviewed recent applications of superquantiles in machine learning, keeping our discussion at a high-level, omitting details, and just providing basic illustrations and pointers to recent research. Second, we have provided explicit expressions of (sub)gradients of (smoothed) superquantiles; here, in contrast, we go down to the details of computation in order to get efficient first-order oracles for superquantile-based functions.
In particular, we have proved that smoothed oracles have essentially the same computational complexity as for the corresponding superquantile functions (Corollary\;\ref{cor:gradSnu} and following discussions).

These fast oracles are implemented in the toolbox\footnote{The code is publicly available at~\url{https://github.com/yassine-laguel/spqr}.} \texttt{spqr} build on top of the popular Python machine learning library~\texttt{scikit-learn}\;\cite{scikit-learn}. This toolbox provides an interface for using standard first-order algorithms; we refer to our numerical experiments of \cite{laguel2020first} and\;\cite{laguel2020device} (see also Appendix\;\ref{sec:a:Numerical Illustration}).
From this experimental experience, we advocate the use of quasi-Newton methods (and in particular L-BGFS; see \eg \cite{nocedal2006numerical}) that gives good performances in practice.

\begin{acknowledgements}
We acknowledge support from ANR-19-P3IA-0003 (MIAI - Grenoble Alpes), NSF DMS 2023166, DMS 1839371, CCF 2019844, the CIFAR program ``Learning in Machines and Brains'', and faculty research awards.
\end{acknowledgements}

\clearpage

\bibliographystyle{abbrvnat}
\bibliography{references}

\begin{thebibliography}{57}
\providecommand{\natexlab}[1]{#1}
\providecommand{\url}[1]{\texttt{#1}}
\expandafter\ifx\csname urlstyle\endcsname\relax
  \providecommand{\doi}[1]{doi: #1}\else
  \providecommand{\doi}{doi: \begingroup \urlstyle{rm}\Url}\fi

\bibitem[Abadi et~al.(2016)Abadi, Barham, Chen, Chen, Davis, Dean, Devin,
  Ghemawat, Irving, Isard, Kudlur, Levenberg, Monga, Moore, Murray, Steiner,
  Tucker, Vasudevan, Warden, Wicke, Yu, and
  Zheng]{DBLP:conf/osdi/AbadiBCCDDDGIIK16}
M.~Abadi, P.~Barham, J.~Chen, Z.~Chen, A.~Davis, J.~Dean, M.~Devin,
  S.~Ghemawat, G.~Irving, M.~Isard, M.~Kudlur, J.~Levenberg, R.~Monga,
  S.~Moore, D.~G. Murray, B.~Steiner, P.~A. Tucker, V.~Vasudevan, P.~Warden,
  M.~Wicke, Y.~Yu, and X.~Zheng.
\newblock Tensorflow: {A} system for large-scale machine learning.
\newblock In K.~Keeton and T.~Roscoe, editors, \emph{12th {USENIX} Symposium on
  Operating Systems Design and Implementation, {OSDI} 2016, Savannah, GA, USA,
  November 2-4, 2016}, pages 265--283. {USENIX} Association, 2016.
\newblock URL
  \url{https://www.usenix.org/conference/osdi16/technical-sessions/presentation/abadi}.

\bibitem[Beck and Teboulle(2012)]{beck2012smoothing}
A.~Beck and M.~Teboulle.
\newblock Smoothing and first order methods: {A} unified framework.
\newblock \emph{{SIAM} J. Optim.}, 22\penalty0 (2):\penalty0 557--580, 2012.
\newblock \doi{10.1137/100818327}.
\newblock URL \url{https://doi.org/10.1137/100818327}.

\bibitem[{Ben-Tal} and {Teboulle}(1986)]{zbMATH04016602}
A.~{Ben-Tal} and M.~{Teboulle}.
\newblock {Expected utility, penalty functions, and duality in stochastic
  nonlinear programming}.
\newblock \emph{{Manage. Sci.}}, 32:\penalty0 1445--1466, 1986.
\newblock ISSN 0025-1909.
\newblock \doi{10.1287/mnsc.32.11.1445}.

\bibitem[{Ben-Tal} and {Teboulle}(2007)]{zbMATH05213441}
A.~{Ben-Tal} and M.~{Teboulle}.
\newblock {An old-new concept of convex risk measures: The optimized certainty
  equivalent}.
\newblock \emph{{Math. Finance}}, 17\penalty0 (3):\penalty0 449--476, 2007.
\newblock ISSN 0960-1627.
\newblock \doi{10.1111/j.1467-9965.2007.00311.x}.

\bibitem[Ben{-}Tal et~al.(2009)Ben{-}Tal, Ghaoui, and
  Nemirovski]{DBLP:books/degruyter/Ben-TalGN09}
A.~Ben{-}Tal, L.~E. Ghaoui, and A.~Nemirovski.
\newblock \emph{Robust Optimization}, volume~28 of \emph{Princeton Series in
  Applied Mathematics}.
\newblock Princeton University Press, 2009.
\newblock ISBN 978-1-4008-3105-0.
\newblock \doi{10.1515/9781400831050}.
\newblock URL \url{https://doi.org/10.1515/9781400831050}.

\bibitem[Chen and Mangasarian(1996)]{DBLP:journals/coap/ChenM96}
C.~Chen and O.~L. Mangasarian.
\newblock A class of smoothing functions for nonlinear and mixed
  complementarity problems.
\newblock \emph{Comput. Optim. Appl.}, 5\penalty0 (2):\penalty0 97--138, 1996.
\newblock \doi{10.1007/BF00249052}.
\newblock URL \url{https://doi.org/10.1007/BF00249052}.

\bibitem[{Cucker} and {Zhou}(2007)]{zbMATH05133436}
F.~{Cucker} and D.~X. {Zhou}.
\newblock \emph{{Learning theory. An approximation theory viewpoint.}},
  volume~24.
\newblock Cambridge: Cambridge University Press, 2007.
\newblock ISBN 978-0-521-86559-3; 978-0-511-27166-3.
\newblock \doi{10.1017/CBO9780511618796}.

\bibitem[Curi et~al.(2020)Curi, Levy, Jegelka, and
  Krause]{DBLP:conf/nips/CuriLJ020}
S.~Curi, K.~Y. Levy, S.~Jegelka, and A.~Krause.
\newblock Adaptive sampling for stochastic risk-averse learning.
\newblock In H.~Larochelle, M.~Ranzato, R.~Hadsell, M.~Balcan, and H.~Lin,
  editors, \emph{Advances in Neural Information Processing Systems 33: Annual
  Conference on Neural Information Processing Systems 2020, NeurIPS 2020,
  December 6-12, 2020, virtual}, 2020.
\newblock URL
  \url{https://proceedings.neurips.cc/paper/2020/hash/0b6ace9e8971cf36f1782aa982a708db-Abstract.html}.

\bibitem[{Dantzig}(1957)]{zbMATH07061852}
G.~B. {Dantzig}.
\newblock {Discrete-variable extremum problems}.
\newblock \emph{{Oper. Res.}}, 5\penalty0 (2):\penalty0 266--288, 1957.
\newblock ISSN 0030-364X.
\newblock \doi{10.1287/opre.5.2.266}.

\bibitem[Duchi and Namkoong(2018)]{DBLP:journals/corr/abs-1810-08750}
J.~C. Duchi and H.~Namkoong.
\newblock Learning models with uniform performance via distributionally robust
  optimization.
\newblock \emph{CoRR}, abs/1810.08750, 2018.
\newblock URL \url{http://arxiv.org/abs/1810.08750}.

\bibitem[Fan et~al.(2017)Fan, Lyu, Ying, and Hu]{DBLP:conf/nips/FanLYH17}
Y.~Fan, S.~Lyu, Y.~Ying, and B.~Hu.
\newblock Learning with average top-k loss.
\newblock In I.~Guyon, U.~von Luxburg, S.~Bengio, H.~M. Wallach, R.~Fergus,
  S.~V.~N. Vishwanathan, and R.~Garnett, editors, \emph{Advances in Neural
  Information Processing Systems 30: Annual Conference on Neural Information
  Processing Systems 2017, December 4-9, 2017, Long Beach, CA, {USA}}, pages
  497--505, 2017.
\newblock URL
  \url{https://proceedings.neurips.cc/paper/2017/hash/6c524f9d5d7027454a783c841250ba71-Abstract.html}.

\bibitem[F{\"{o}}llmer and Schied(2002)]{DBLP:journals/fs/FollmerS02}
H.~F{\"{o}}llmer and A.~Schied.
\newblock Convex measures of risk and trading constraints.
\newblock \emph{Finance Stochastics}, 6\penalty0 (4):\penalty0 429--447, 2002.
\newblock \doi{10.1007/s007800200072}.
\newblock URL \url{https://doi.org/10.1007/s007800200072}.

\bibitem[Guigues and Sagastiz{\'{a}}bal(2013)]{DBLP:journals/mp/GuiguesS13}
V.~Guigues and C.~A. Sagastiz{\'{a}}bal.
\newblock Risk-averse feasible policies for large-scale multistage stochastic
  linear programs.
\newblock \emph{Math. Program.}, 138\penalty0 (1-2):\penalty0 167--198, 2013.
\newblock \doi{10.1007/s10107-012-0592-1}.
\newblock URL \url{https://doi.org/10.1007/s10107-012-0592-1}.

\bibitem[Hiriart-Urruty and {Lemar\'echal}(1993)]{hull}
J.-B. Hiriart-Urruty and C.~{Lemar\'echal}.
\newblock \emph{Convex Analysis and Min\-im\-iz\-a\-tion Alg\-or\-ithms}.
\newblock Springer Verlag, Heidelberg, 1993.
\newblock Two volumes.

\bibitem[Ho{-}Nguyen and Wright(2020)]{DBLP:journals/corr/abs-2005-13815}
N.~Ho{-}Nguyen and S.~J. Wright.
\newblock Adversarial classification via distributional robustness with
  wasserstein ambiguity.
\newblock \emph{CoRR}, abs/2005.13815, 2020.
\newblock URL \url{https://arxiv.org/abs/2005.13815}.

\bibitem[Holstein et~al.(2019)Holstein, Vaughan, III, Dud{\'{\i}}k, and
  Wallach]{DBLP:conf/chi/HolsteinVDDW19}
K.~Holstein, J.~W. Vaughan, H.~D. III, M.~Dud{\'{\i}}k, and H.~M. Wallach.
\newblock Improving fairness in machine learning systems: What do industry
  practitioners need?
\newblock In S.~A. Brewster, G.~Fitzpatrick, A.~L. Cox, and V.~Kostakos,
  editors, \emph{Proceedings of the 2019 {CHI} Conference on Human Factors in
  Computing Systems, {CHI} 2019, Glasgow, Scotland, UK, May 04-09, 2019}, page
  600. {ACM}, 2019.
\newblock \doi{10.1145/3290605.3300830}.
\newblock URL \url{https://doi.org/10.1145/3290605.3300830}.

\bibitem[{Howard} and {Matheson}(1972)]{zbMATH03378669}
R.~A. {Howard} and J.~E. {Matheson}.
\newblock {Risk-sensitive Markov decision processes}.
\newblock \emph{{Manage. Sci., Theory}}, 18:\penalty0 356--369, 1972.
\newblock \doi{10.1287/mnsc.18.7.356}.

\bibitem[Kairouz et~al.(2021)Kairouz, McMahan, Avent, Bellet, Bennis, Bhagoji,
  Bonawitz, Charles, Cormode, Cummings, D'Oliveira, Eichner, Rouayheb, Evans,
  Gardner, Garrett, Gasc{\'{o}}n, Ghazi, Gibbons, Gruteser, Harchaoui, He, He,
  Huo, Hutchinson, Hsu, Jaggi, Javidi, Joshi, Khodak, Kone{\v{c}}n{\'y},
  Korolova, Koushanfar, Koyejo, Lepoint, Liu, Mittal, Mohri, Nock,
  {\"{O}}zg{\"{u}}r, Pagh, Qi, Ramage, Raskar, Raykova, Song, Song, Stich, Sun,
  Suresh, Tram{\`{e}}r, Vepakomma, Wang, Xiong, Xu, Yang, Yu, Yu, and
  Zhao]{DBLP:journals/ftml/KairouzMABBBBCC21}
P.~Kairouz, H.~B. McMahan, B.~Avent, A.~Bellet, M.~Bennis, A.~N. Bhagoji, K.~A.
  Bonawitz, Z.~Charles, G.~Cormode, R.~Cummings, R.~G.~L. D'Oliveira,
  H.~Eichner, S.~E. Rouayheb, D.~Evans, J.~Gardner, Z.~Garrett,
  A.~Gasc{\'{o}}n, B.~Ghazi, P.~B. Gibbons, M.~Gruteser, Z.~Harchaoui, C.~He,
  L.~He, Z.~Huo, B.~Hutchinson, J.~Hsu, M.~Jaggi, T.~Javidi, G.~Joshi,
  M.~Khodak, J.~Kone{\v{c}}n{\'y}, A.~Korolova, F.~Koushanfar, S.~Koyejo,
  T.~Lepoint, Y.~Liu, P.~Mittal, M.~Mohri, R.~Nock, A.~{\"{O}}zg{\"{u}}r,
  R.~Pagh, H.~Qi, D.~Ramage, R.~Raskar, M.~Raykova, D.~Song, W.~Song, S.~U.
  Stich, Z.~Sun, A.~T. Suresh, F.~Tram{\`{e}}r, P.~Vepakomma, J.~Wang,
  L.~Xiong, Z.~Xu, Q.~Yang, F.~X. Yu, H.~Yu, and S.~Zhao.
\newblock Advances and open problems in federated learning.
\newblock \emph{Found. Trends Mach. Learn.}, 14\penalty0 (1-2):\penalty0
  1--210, 2021.
\newblock \doi{10.1561/2200000083}.
\newblock URL \url{https://doi.org/10.1561/2200000083}.

\bibitem[Kamishima et~al.(2012)Kamishima, Akaho, Asoh, and
  Sakuma]{DBLP:conf/pkdd/KamishimaAAS12}
T.~Kamishima, S.~Akaho, H.~Asoh, and J.~Sakuma.
\newblock Fairness-aware classifier with prejudice remover regularizer.
\newblock In P.~A. Flach, T.~D. Bie, and N.~Cristianini, editors, \emph{Machine
  Learning and Knowledge Discovery in Databases - European Conference, {ECML}
  {PKDD} 2012, Bristol, UK, September 24-28, 2012. Proceedings, Part {II}},
  volume 7524 of \emph{Lecture Notes in Computer Science}, pages 35--50.
  Springer, 2012.
\newblock \doi{10.1007/978-3-642-33486-3\_3}.
\newblock URL \url{https://doi.org/10.1007/978-3-642-33486-3\_3}.

\bibitem[Kawaguchi and Lu(2020)]{DBLP:conf/aistats/KawaguchiL20}
K.~Kawaguchi and H.~Lu.
\newblock Ordered {SGD:} {A} new stochastic optimization framework for
  empirical risk minimization.
\newblock In S.~Chiappa and R.~Calandra, editors, \emph{The 23rd International
  Conference on Artificial Intelligence and Statistics, {AISTATS} 2020, 26-28
  August 2020, Online [Palermo, Sicily, Italy]}, volume 108 of
  \emph{Proceedings of Machine Learning Research}, pages 669--679. {PMLR},
  2020.
\newblock URL \url{http://proceedings.mlr.press/v108/kawaguchi20a.html}.

\bibitem[Knight(2018)]{knight2018selfdriving}
W.~Knight.
\newblock A self-driving {Uber} has killed a pedestrian in {Arizona}.
\newblock \emph{Ethical Tech}, March 2018.

\bibitem[Laguel et~al.(2020)Laguel, Malick, and Harchaoui]{laguel2020first}
Y.~Laguel, J.~Malick, and Z.~Harchaoui.
\newblock First-order optimization for superquantile-based supervised learning.
\newblock In \emph{30th {IEEE} International Workshop on Machine Learning for
  Signal Processing, {MLSP} 2020, Espoo, Finland, September 21-24, 2020}, pages
  1--6. {IEEE}, 2020.
\newblock \doi{10.1109/MLSP49062.2020.9231909}.
\newblock URL \url{https://doi.org/10.1109/MLSP49062.2020.9231909}.

\bibitem[Laguel et~al.(2021)Laguel, Pillutla, Malick, and
  Harchaoui]{laguel2020device}
Y.~Laguel, K.~Pillutla, J.~Malick, and Z.~Harchaoui.
\newblock A superquantile approach to federated learning with heterogeneous
  devices.
\newblock In \emph{55th Annual Conference on Information Sciences and Systems,
  {CISS} 2021, Baltimore, MD, USA, March 24-26, 2021}, pages 1--6. {IEEE},
  2021.
\newblock \doi{10.1109/CISS50987.2021.9400318}.
\newblock URL \url{https://doi.org/10.1109/CISS50987.2021.9400318}.

\bibitem[Lee et~al.(2020)Lee, Park, and Shin]{lee2020learning}
J.~Lee, S.~Park, and J.~Shin.
\newblock Learning bounds for risk-sensitive learning.
\newblock In H.~Larochelle, M.~Ranzato, R.~Hadsell, M.~Balcan, and H.~Lin,
  editors, \emph{Advances in Neural Information Processing Systems 33: Annual
  Conference on Neural Information Processing Systems 2020, NeurIPS 2020,
  December 6-12, 2020, virtual}, 2020.
\newblock URL
  \url{https://proceedings.neurips.cc/paper/2020/hash/9f60ab2b55468f104055b16df8f69e81-Abstract.html}.

\bibitem[Levy et~al.(2020)Levy, Carmon, Duchi, and Sidford]{levy2020large}
D.~Levy, Y.~Carmon, J.~C. Duchi, and A.~Sidford.
\newblock Large-scale methods for distributionally robust optimization.
\newblock In H.~Larochelle, M.~Ranzato, R.~Hadsell, M.~Balcan, and H.~Lin,
  editors, \emph{Advances in Neural Information Processing Systems 33: Annual
  Conference on Neural Information Processing Systems 2020, NeurIPS 2020,
  December 6-12, 2020, virtual}, 2020.
\newblock URL
  \url{https://proceedings.neurips.cc/paper/2020/hash/64986d86a17424eeac96b08a6d519059-Abstract.html}.

\bibitem[Luna et~al.(2016)Luna, Sagastiz{\'{a}}bal, and
  Solodov]{luna2016approximation}
J.~P. Luna, C.~A. Sagastiz{\'{a}}bal, and M.~V. Solodov.
\newblock An approximation scheme for a class of risk-averse stochastic
  equilibrium problems.
\newblock \emph{Math. Program.}, 157\penalty0 (2):\penalty0 451--481, 2016.
\newblock \doi{10.1007/s10107-016-0988-4}.
\newblock URL \url{https://doi.org/10.1007/s10107-016-0988-4}.

\bibitem[Metz(2018)]{metz2018microsoft}
R.~Metz.
\newblock Microsoft's neo-{Nazi} sexbot was a great lesson for makers of {AI}
  assistants.
\newblock \emph{Artificial Intelligence}, March 2018.

\bibitem[Mhammedi et~al.(2020)Mhammedi, Guedj, and Williamson]{mhammedi2020pac}
Z.~Mhammedi, B.~Guedj, and R.~C. Williamson.
\newblock Pac-bayesian bound for the conditional value at risk.
\newblock In H.~Larochelle, M.~Ranzato, R.~Hadsell, M.~Balcan, and H.~Lin,
  editors, \emph{Advances in Neural Information Processing Systems 33: Annual
  Conference on Neural Information Processing Systems 2020, NeurIPS 2020,
  December 6-12, 2020, virtual}, 2020.
\newblock URL
  \url{https://proceedings.neurips.cc/paper/2020/hash/d02e9bdc27a894e882fa0c9055c99722-Abstract.html}.

\bibitem[Miranda(2014)]{miranda2014superquantile}
S.~I. Miranda.
\newblock Superquantile regression: theory, algorithms, and applications.
\newblock Technical report, Naval postgraduate school Monterey ca, 2014.

\bibitem[Morimura et~al.(2010)Morimura, Sugiyama, Kashima, Hachiya, and
  Tanaka]{morimura2010nonparametric}
T.~Morimura, M.~Sugiyama, H.~Kashima, H.~Hachiya, and T.~Tanaka.
\newblock Nonparametric return distribution approximation for reinforcement
  learning.
\newblock In J.~F{\"{u}}rnkranz and T.~Joachims, editors, \emph{Proceedings of
  the 27th International Conference on Machine Learning (ICML-10), June 21-24,
  2010, Haifa, Israel}, pages 799--806. Omnipress, 2010.
\newblock URL \url{https://icml.cc/Conferences/2010/papers/652.pdf}.

\bibitem[{Nesterov} and {Spokoiny}(2017)]{nesterov2017random}
Y.~{Nesterov} and V.~{Spokoiny}.
\newblock {Random gradient-free minimization of convex functions}.
\newblock \emph{{Found. Comput. Math.}}, 17\penalty0 (2):\penalty0 527--566,
  2017.
\newblock ISSN 1615-3375.
\newblock \doi{10.1007/s10208-015-9296-2}.

\bibitem[Nesterov(2004)]{nesterov-book}
Y.~E. Nesterov.
\newblock \emph{Introductory Lectures on Convex Optimization - {A} Basic
  Course}, volume~87 of \emph{Applied Optimization}.
\newblock Springer, 2004.
\newblock ISBN 978-1-4613-4691-3.
\newblock \doi{10.1007/978-1-4419-8853-9}.
\newblock URL \url{https://doi.org/10.1007/978-1-4419-8853-9}.

\bibitem[Nesterov(2005)]{nesterov2005smooth}
Y.~E. Nesterov.
\newblock Smooth minimization of non-smooth functions.
\newblock \emph{Math. Program.}, 103\penalty0 (1):\penalty0 127--152, 2005.
\newblock \doi{10.1007/s10107-004-0552-5}.
\newblock URL \url{https://doi.org/10.1007/s10107-004-0552-5}.

\bibitem[{Nocedal} and {Wright}(2006)]{nocedal2006numerical}
J.~{Nocedal} and S.~J. {Wright}.
\newblock \emph{{Numerical optimization}}.
\newblock New York, NY: Springer, 2006.
\newblock ISBN 0-387-30303-0.

\bibitem[Paszke et~al.(2019)Paszke, Gross, Massa, Lerer, Bradbury, Chanan,
  Killeen, Lin, Gimelshein, Antiga, Desmaison, Kopf, Yang, DeVito, Raison,
  Tejani, Chilamkurthy, Steiner, Fang, Bai, and Chintala]{paszke2017automatic}
A.~Paszke, S.~Gross, F.~Massa, A.~Lerer, J.~Bradbury, G.~Chanan, T.~Killeen,
  Z.~Lin, N.~Gimelshein, L.~Antiga, A.~Desmaison, A.~Kopf, E.~Yang, Z.~DeVito,
  M.~Raison, A.~Tejani, S.~Chilamkurthy, B.~Steiner, L.~Fang, J.~Bai, and
  S.~Chintala.
\newblock Pytorch: An imperative style, high-performance deep learning library.
\newblock In H.~Wallach, H.~Larochelle, A.~Beygelzimer, F.~d\textquotesingle
  Alch\'{e}-Buc, E.~Fox, and R.~Garnett, editors, \emph{Advances in Neural
  Information Processing Systems 32}, pages 8024--8035. Curran Associates,
  Inc., 2019.
\newblock URL
  \url{http://papers.neurips.cc/paper/9015-pytorch-an-imperative-style-high-performance-deep-learning-library.pdf}.

\bibitem[Pedregosa et~al.(2011)Pedregosa, Varoquaux, Gramfort, Michel, Thirion,
  Grisel, Blondel, Prettenhofer, Weiss, Dubourg, VanderPlas, Passos,
  Cournapeau, Brucher, Perrot, and Duchesnay]{scikit-learn}
F.~Pedregosa, G.~Varoquaux, A.~Gramfort, V.~Michel, B.~Thirion, O.~Grisel,
  M.~Blondel, P.~Prettenhofer, R.~Weiss, V.~Dubourg, J.~VanderPlas, A.~Passos,
  D.~Cournapeau, M.~Brucher, M.~Perrot, and E.~Duchesnay.
\newblock Scikit-learn: Machine learning in python.
\newblock \emph{J. Mach. Learn. Res.}, 12:\penalty0 2825--2830, 2011.
\newblock URL \url{http://dl.acm.org/citation.cfm?id=2078195}.

\bibitem[{Pollard}(2002)]{pollard2002user}
D.~{Pollard}.
\newblock \emph{{A user's guide to measure theoretic probability}}, volume~8.
\newblock Cambridge: Cambridge University Press, 2002.
\newblock ISBN 0-521-80242-3; 0-521-00289-3.
\newblock \doi{10.1017/CBO9780511811555}.

\bibitem[Recht et~al.(2019)Recht, Roelofs, Schmidt, and
  Shankar]{recht2019imagenet}
B.~Recht, R.~Roelofs, L.~Schmidt, and V.~Shankar.
\newblock Do imagenet classifiers generalize to imagenet?
\newblock In K.~Chaudhuri and R.~Salakhutdinov, editors, \emph{Proceedings of
  the 36th International Conference on Machine Learning, {ICML} 2019, 9-15 June
  2019, Long Beach, California, {USA}}, volume~97 of \emph{Proceedings of
  Machine Learning Research}, pages 5389--5400. {PMLR}, 2019.
\newblock URL \url{http://proceedings.mlr.press/v97/recht19a.html}.

\bibitem[{Rockafellar}(2018)]{rockafellar2018solving}
R.~T. {Rockafellar}.
\newblock {Solving stochastic programming problems with risk measures by
  progressive hedging}.
\newblock \emph{{Set-Valued Var. Anal.}}, 26\penalty0 (4):\penalty0 759--768,
  2018.
\newblock ISSN 1877-0533.
\newblock \doi{10.1007/s11228-017-0437-4}.

\bibitem[Rockafellar and Royset(2013)]{rockafellar2013superquantiles}
R.~T. Rockafellar and J.~O. Royset.
\newblock Superquantiles and their applications to risk, random variables, and
  regression.
\newblock In \emph{Theory Driven by Influential Applications}, pages 151--167.
  INFORMS, 2013.

\bibitem[Rockafellar and Royset(2014)]{rockafellar2014random}
R.~T. Rockafellar and J.~O. Royset.
\newblock Random variables, monotone relations, and convex analysis.
\newblock \emph{Math. Program.}, 148\penalty0 (1-2):\penalty0 297--331, 2014.
\newblock \doi{10.1007/s10107-014-0801-1}.
\newblock URL \url{https://doi.org/10.1007/s10107-014-0801-1}.

\bibitem[Rockafellar and Uryasev(2002)]{rockafellar2002conditional}
R.~T. Rockafellar and S.~Uryasev.
\newblock Conditional value-at-risk for general loss distributions.
\newblock \emph{Journal of banking \& finance}, 26\penalty0 (7):\penalty0
  1443--1471, 2002.

\bibitem[Rockafellar and Wets(2009)]{rockafellar2009variational}
R.~T. Rockafellar and R.~J.-B. Wets.
\newblock \emph{Variational analysis}, volume 317.
\newblock Springer Science \& Business Media, 2009.

\bibitem[Rockafellar et~al.(2000)Rockafellar, Uryasev,
  et~al.]{rockafellar2000optimization}
R.~T. Rockafellar, S.~Uryasev, et~al.
\newblock Optimization of conditional value-at-risk.
\newblock \emph{Journal of risk}, 2:\penalty0 21--42, 2000.

\bibitem[Rockafellar et~al.(2014)Rockafellar, Royset, and
  Miranda]{rockafellar2014superquantile}
R.~T. Rockafellar, J.~O. Royset, and S.~I. Miranda.
\newblock Superquantile regression with applications to buffered reliability,
  uncertainty quantification, and conditional value-at-risk.
\newblock \emph{Eur. J. Oper. Res.}, 234\penalty0 (1):\penalty0 140--154, 2014.
\newblock \doi{10.1016/j.ejor.2013.10.046}.
\newblock URL \url{https://doi.org/10.1016/j.ejor.2013.10.046}.

\bibitem[Ruszczynski and Shapiro(2006)]{ruszczynski2006optimization}
A.~Ruszczynski and A.~Shapiro.
\newblock Optimization of convex risk functions.
\newblock \emph{Math. Oper. Res.}, 31\penalty0 (3):\penalty0 433--452, 2006.
\newblock \doi{10.1287/moor.1050.0186}.
\newblock URL \url{https://doi.org/10.1287/moor.1050.0186}.

\bibitem[Sarykalin et~al.(2008)Sarykalin, Serraino, and
  Uryasev]{sarykalin2008value}
S.~Sarykalin, G.~Serraino, and S.~Uryasev.
\newblock Value-at-risk vs. conditional value-at-risk in risk management and
  optimization.
\newblock In \emph{State-of-the-art decision-making tools in the
  information-intensive age}, pages 270--294. Informs, 2008.

\bibitem[Shafieezadeh{-}Abadeh et~al.(2019)Shafieezadeh{-}Abadeh, Kuhn, and
  Esfahani]{shafieezadeh2019regularization}
S.~Shafieezadeh{-}Abadeh, D.~Kuhn, and P.~M. Esfahani.
\newblock Regularization via mass transportation.
\newblock \emph{J. Mach. Learn. Res.}, 20:\penalty0 103:1--103:68, 2019.
\newblock URL \url{http://jmlr.org/papers/v20/17-633.html}.

\bibitem[Shalev{-}Shwartz and Ben{-}David(2014)]{shalev2014understanding}
S.~Shalev{-}Shwartz and S.~Ben{-}David.
\newblock \emph{Understanding Machine Learning - From Theory to Algorithms}.
\newblock Cambridge University Press, 2014.
\newblock ISBN 978-1-10-705713-5.
\newblock URL
  \url{http://www.cambridge.org/de/academic/subjects/computer-science/pattern-recognition-and-machine-learning/understanding-machine-learning-theory-algorithms}.

\bibitem[Shapiro et~al.(2014)Shapiro, Dentcheva, and
  Ruszczynski]{Shapiro:2014:LSP:2678054}
A.~Shapiro, D.~Dentcheva, and A.~Ruszczynski.
\newblock \emph{Lectures on Stochastic Programming - Modeling and Theory,
  Second Edition}, volume~16 of \emph{{MOS-SIAM} Series on Optimization}.
\newblock {SIAM}, 2014.
\newblock ISBN 978-1-61197-342-6.
\newblock URL \url{http://bookstore.siam.org/mo16/}.

\bibitem[Soma and Yoshida(2020)]{soma2020statistical}
T.~Soma and Y.~Yoshida.
\newblock Statistical learning with conditional value at risk.
\newblock \emph{CoRR}, abs/2002.05826, 2020.
\newblock URL \url{https://arxiv.org/abs/2002.05826}.

\bibitem[{Sutton} and {Barto}(2018)]{sutton2018reinforcement}
R.~S. {Sutton} and A.~G. {Barto}.
\newblock \emph{{Reinforcement learning. An introduction}}.
\newblock Cambridge, MA: MIT Press, 2018.
\newblock ISBN 978-0-262-03924-6.

\bibitem[Sutton et~al.(1999)Sutton, McAllester, Singh, and
  Mansour]{sutton2000policy}
R.~S. Sutton, D.~A. McAllester, S.~P. Singh, and Y.~Mansour.
\newblock Policy gradient methods for reinforcement learning with function
  approximation.
\newblock In S.~A. Solla, T.~K. Leen, and K.~M{\"{u}}ller, editors,
  \emph{Advances in Neural Information Processing Systems 12, {[NIPS}
  Conference, Denver, Colorado, USA, November 29 - December 4, 1999]}, pages
  1057--1063. The {MIT} Press, 1999.
\newblock URL
  \url{http://papers.nips.cc/paper/1713-policy-gradient-methods-for-reinforcement-learning-with-function-approximation}.

\bibitem[Tamar et~al.(2015)Tamar, Chow, Ghavamzadeh, and
  Mannor]{tamar2015policy}
A.~Tamar, Y.~Chow, M.~Ghavamzadeh, and S.~Mannor.
\newblock Policy gradient for coherent risk measures.
\newblock In C.~Cortes, N.~D. Lawrence, D.~D. Lee, M.~Sugiyama, and R.~Garnett,
  editors, \emph{Advances in Neural Information Processing Systems 28: Annual
  Conference on Neural Information Processing Systems 2015, December 7-12,
  2015, Montreal, Quebec, Canada}, pages 1468--1476, 2015.
\newblock URL
  \url{https://proceedings.neurips.cc/paper/2015/hash/024d7f84fff11dd7e8d9c510137a2381-Abstract.html}.

\bibitem[{Vershynin}(2018)]{vershynin2018high}
R.~{Vershynin}.
\newblock \emph{{High-dimensional probability. An introduction with
  applications in data science}}, volume~47.
\newblock Cambridge: Cambridge University Press, 2018.
\newblock ISBN 978-1-108-41519-4; 978-1-108-23159-6.
\newblock \doi{10.1017/9781108231596}.

\bibitem[{Wainwright}(2019)]{wainwright_2019}
M.~J. {Wainwright}.
\newblock \emph{{High-dimensional statistics. A non-asymptotic viewpoint}},
  volume~48.
\newblock Cambridge: Cambridge University Press, 2019.
\newblock ISBN 978-1-108-49802-9; 978-1-108-62777-1.
\newblock \doi{10.1017/9781108627771}.

\bibitem[Williamson and Menon(2019)]{willaimson2019fairness}
R.~C. Williamson and A.~K. Menon.
\newblock Fairness risk measures.
\newblock In K.~Chaudhuri and R.~Salakhutdinov, editors, \emph{Proceedings of
  the 36th International Conference on Machine Learning, {ICML} 2019, 9-15 June
  2019, Long Beach, California, {USA}}, volume~97 of \emph{Proceedings of
  Machine Learning Research}, pages 6786--6797. {PMLR}, 2019.
\newblock URL \url{http://proceedings.mlr.press/v97/williamson19a.html}.

\end{thebibliography}

\appendix
\section{Proof of Theorem\texorpdfstring{\;\ref{thm:uniform-convergence-param}}{}} \label{sec:a:gen_bound}

In this appendix, we provide a complete proof of Theorem\;\ref{thm:uniform-convergence-param}. For classical results in this spirit, we refer to the monograph \cite{zbMATH05133436}.
For discussions on statistical aspects of statistical learning, we refer to \eg\cite{DBLP:journals/corr/abs-1810-08750,mhammedi2020pac,lee2020learning}.


    The key step in the proof of Theorem\;\ref{thm:uniform-convergence-param} is to show the uniform convergence
    \begin{equation}\label{eq:key}
    \text{$\Rpn(w) \to \Rp(w)$ almost surely for all $w \in W$.}
    \end{equation}
    Indeed, once we have this,
    the result immediately follows as
	\begin{align}
		0 \le \Rp(w_n^\star) - \Rp(w^\star)
		&=
		\Rp(w_n^\star) - \Rpn(w_n^\star)
		+ \Rpn(w_n^\star) - \Rpn(w^\star)
		+ \Rpn(w^\star) - \Rp(w^\star)\\
		&\le
		2 \sup_{w \in W} |\Rpn(w) - \Rp(w)|
		\to 0,
	\end{align}
	where we use $\Rpn(w_n^\star) \le \Rpn(w^\star)$ in the second inequality.

	In order to prove~\eqref{eq:key}, 
	we use the variational expression of the superquantile\;\eqref{eq:def_min_cvar}.
	We define
	\[
		\bar \Rp(w, \eta)
		= \eta + \frac{1}{1-p} \mathbb{E}_{(x, y)\sim P}
		[\max(\ell(y, \varphi(w, x))-\eta,0)] \,,
	\]
	so that, using that the loss is bounded by $B$, we can write
	\[
	\Rp(w) = \min_{\eta \in [0, B]} \bar \Rp(w, \eta).
	\]
	We define the analogous empirical version $\bar \Rpn(w, \eta)$ so that
	$\Rpn(w) = \min_{\eta \in [0, B]} \bar \Rpn(w, \eta)$.
	Note that $\bar \Rpn(w, \eta)$ is measurable for each fixed $(w, \eta)$ and $\Rpn(w)$ is measurable for each fixed $w$.

	\textit{Claim 1: Under Assumption\;\ref{assump:loss}, the random variable
	\[
	\delta_n(w, \eta) := \bRpn(w, \eta) - \bar \Rp(w, \eta)
	\]
	has mean zero, lies almost surely in $[-B, B]$, and satisfies}
	\begin{align} \label{eq:lipschitz}
		|\delta_n(w, \eta) - \delta_n(w', \eta')|
		\le 2M/(1-p) \dist_\varphi(w, w') + 2(1+1/(1-p)) |\eta - \eta'| \,.
	\end{align}
	Note first that $\mathbb{E}[\bRpn(w, \eta)] = \bar \Rp(w, \eta)$
	and that the boundedness of $\delta_n$ comes from the boundedness of the loss function.
	The Lipschitzness of $\delta_n$ also comes from the one of the loss function, as follows.
	Using that $\max\{\cdot, 0\}$ is $1$-Lipschitz and that the loss $\ell$ is
	$M$-Lipschitz, we get
	\begin{align}
		|\max\{\ell(y, \varphi(w, x)) - \eta, 0\}
			-& \max\{\ell(y, \varphi(w', x)) - \eta', 0\}|
		\\
		&\le |\ell(y, \varphi(w, x)) - \ell(y, \varphi(w', x))|
		 	+ |\eta - \eta'| \\
        &\le M \|\varphi(w, x) - \varphi(w', x)\| + |\eta - \eta'| \\
		&\le M\dist_\varphi(w, w') + |\eta - \eta'| \,.
	\end{align}
	Then, \eqref{eq:lipschitz} simply follows from the triangle inequality, and Claim 1 is proved.

   \medskip
    The next step in the proof is, for a given  $\varepsilon > 0$
    \begin{itemize}
        \item to construct a cover $T$ of $W \times [0, B]$, and then
        \medskip
        \item to control the convergence over the points of $T$, more precisely to control the probability of the event
        \[
		E_n(\varepsilon) = \bigcap_{(w, \eta) \in T} \left\{ \delta_n(w, \eta) \le \varepsilon / 2 \right\} \,.
	    \]
    \end{itemize}  

    First, using Assumption\;\ref{assump:cover}, we consider
	$T_1$ a $(\varepsilon(1-p)/(8M))$-cover of $W$ with
	respect to $\dist_\varphi$. We also consider $T_2$ a uniform discretization of the line segment $[0, B]$ at width $\varepsilon(1+1/(1-p))/8$. We can introduce the cover of $W \times [0, B]$
	\[
	T = T_1 \times T_2 \subset W \times [0, B]. 
	\]
	Since, $|T_2| = 8B/((1+1/(1-p))\varepsilon)$, we have that
	$|T| = (8B/((1+1/(1-p))\varepsilon)) \covnum(\varepsilon(1-p)/(8M))$.
	 Note that the event $\left\{ \delta_n(w, \eta) \le \varepsilon / 2 \right\}$ for fixed $(w, \eta)$ since $\delta_n(w, \eta)$ is measurable, and therefore,
    $E_n(\varepsilon)$ is measurable since it is a finite intersection.

    To get uniform convergence, it is sufficient to control what happens at points of\;$T$. Indeed, for any $(w, \eta)$, there exists a point $(w', \eta') \in T$ such that $\dist_\varphi(w, w') \le \varepsilon(1-p)/(8M)$ and $|\eta - \eta'| \le \varepsilon (1+1/(1-p))/ 8$. As a consequence, if the event $E_n(\varepsilon)$ holds, then
	\begin{align}
		\delta_n(w, \eta)
		&\ \le \delta_n(w', \eta') +
			|\delta_n(w, \eta) - \delta_n(w', \eta')| \\
		&\stackrel{\eqref{eq:lipschitz}}{\le}
			\delta_n(w', \eta') +
			2M/(1-p) \dist_\varphi(w, w') + 2(1+1/(1-p)) | \eta - \eta'|\\
		&\ \le \frac{\varepsilon}{2} + \frac{\varepsilon}{4} + \frac{\varepsilon}{4}
		= \varepsilon.
	\end{align}

	This implies that events of interest are included in
	$\overline{E}_n(\varepsilon)$, the complement of $E_n(\varepsilon)$; we have indeed
	\[
	\left\{\sup_{w \in W} \, |\Rpn(w) - \Rp(w)| > \varepsilon\right\}
	\subset
	\left\{\sup_{(w,\eta) \in W \times [0, B]} \,  \delta_n(w, \eta) >\varepsilon\right\}
	\subset \overline E_n(\varepsilon)\,.
	\]
	Postponing the proof of measurability of these events to Claim 3 at the end of this proof, we have the following bound on the sum of probabilities
	\begin{equation}\label{eq:finite}
	    \sum_{n=1}^\infty
	    \mathbb{P}\Big(\sup_{w \in W }
	    \, |\Rpn(w) - \Rp(w)| > \varepsilon\Big)
	    \le
	    \sum_{n=1}^\infty \mathbb{P}\big(\overline E_n(\varepsilon)\big).
	\end{equation}

    \textit{Claim 2: The probabilities of the complements of $E_n(\varepsilon)$ are summable, i.e.,}
    \[
	\sum_{n=1}^\infty \mathbb{P}\big(\overline E_n(\varepsilon)\big)  < \infty \,.
	\]
	This is a direct application of the
	Hoeffding's inequality (see e.g.~\cite[Theorem 2.2.2]{vershynin2018high})
	as follows. For any fixed $(w, \eta) \in W \times [0, B]$, the Hoeffding's inequality gives
	\[
		\mathbb{P}(|\delta_n(w, \eta)| > \varepsilon/2) \le
		2 \exp\left(- \frac{n\varepsilon^2}{2B^2}\right) \,.
	\]
	Applied to 
	all $(w, \eta) \in T$, this yields
	\[
		\mathbb{P}\big(\overline E_n(\varepsilon)\big) \le
		2|T|\exp\left(- \frac{n\varepsilon^2}{2B^2}\right)
		= \frac{16B}{((1+1/(1-p))\varepsilon} \,
		\covnum\left(\frac{\varepsilon(1-p)}{8M}\right) \exp\left(- \frac{n\varepsilon^2}{2B^2}\right) \,.
	\]
    and proves Claim 2.

\medskip

We conclude on the uniform convergence \eqref{eq:key} with the Borel-Cantelli Lemma by the classical rationale (see e.g.\;the textbook \cite[Chap.\;2, Sec.\;6]{pollard2002user}): the bound \eqref{eq:finite} and Claim 2 give that the probabilities for any $\epsilon$ are summable; applying Borel-Cantelli with the sequence $\epsilon_k=1/k$ gives the uniform convergence \eqref{eq:key}, which completes the proof of the theorem.

Finally, it remains to show measurability of some events of interest.

\textit{Claim 3: The following events are measurable for each $\varepsilon > 0$:}
\begin{align}
	E'_n(\varepsilon) &:= \left\{\sup_{w \in W} \, |\Rpn(w) - \Rp(w)| > \varepsilon\right\} \,,  \\
    E''_n(\varepsilon) &:=
	\left\{\sup_{(w,\eta) \in W \times [0, B]} \,  \delta_n(w, \eta) >\varepsilon\right\} \,.
\end{align}

We prove the claim for $E'_n(\varepsilon)$ and the second one is entirely analogous. Since
the set $\mathbb{Q}^d$ of $d$-dimensional rationals is dense in $\R^d$ and the map $w \mapsto |\Rpn(w) - \Rp(w)|$ is continuous, we have that
\[
    \sup_{w \in W} \, |\Rpn(w) - \Rp(w)| =
    \sup_{w \in W \cap \mathbb{Q}^d} \, |\Rpn(w) - \Rp(w)| \,.
\]
Since the latter term is a supremum
over a countable set of measurable random variables, we get that $E'_n(\varepsilon)$ is measurable.

\section{Numerical Illustrations} \label{sec:a:Numerical Illustration}

We provide simple illustrations of the interest of using superquantile for machine learning. More precisely, we reproduce the experimental framework of the computational experiments of \cite{DBLP:conf/nips/CuriLJ020} and we solve the superquantile optimization problems with the approach depicted here, by combining smoothing and quasi-Newton. For additional experiments with other datasets, metrics, and contexts, we refer to \cite{DBLP:conf/nips/CuriLJ020}.

We consider two basic machine learning tasks (regression and classification) with linear prediction functions $\varphi(\x,x )=\trans{\x}x$ and with two standard datasets, from the UCI ML repository. Denoting these datasets $\Pn= \{(x_i,y_i)\}_{1 \leq i \leq n}$, we introduce the (regularized) empirical risk minimization
\[
\min_{\x \in \Rd}
~~\mathbb{E}_{(x,y)\sim\Pn}\left[\ell(y, \trans{\x}x)\right] 
+\frac{1}{2n}\|\x\|^2\, ,
\]
and its smoothed superquantile analogous
\[
\min_{\x \in \Rd}
~~{[\Spnu]}_{(x,y)\sim\Pn}\left[\ell(y, \trans{\x}x) \right]
+\frac{1}{2n}\|\x\|^2\, .
\]

We solve these problems using L-BFGS via the toolbox \texttt{SPQR}~\cite{laguel2020first} offering an simple user-interface and implementing the oracles (with the Euclidean smoothing of Example\;\ref{ex:l2} for the smoothed approximation).

\begin{figure}[h!]
\begin{center}
  \includegraphics[width=7cm]{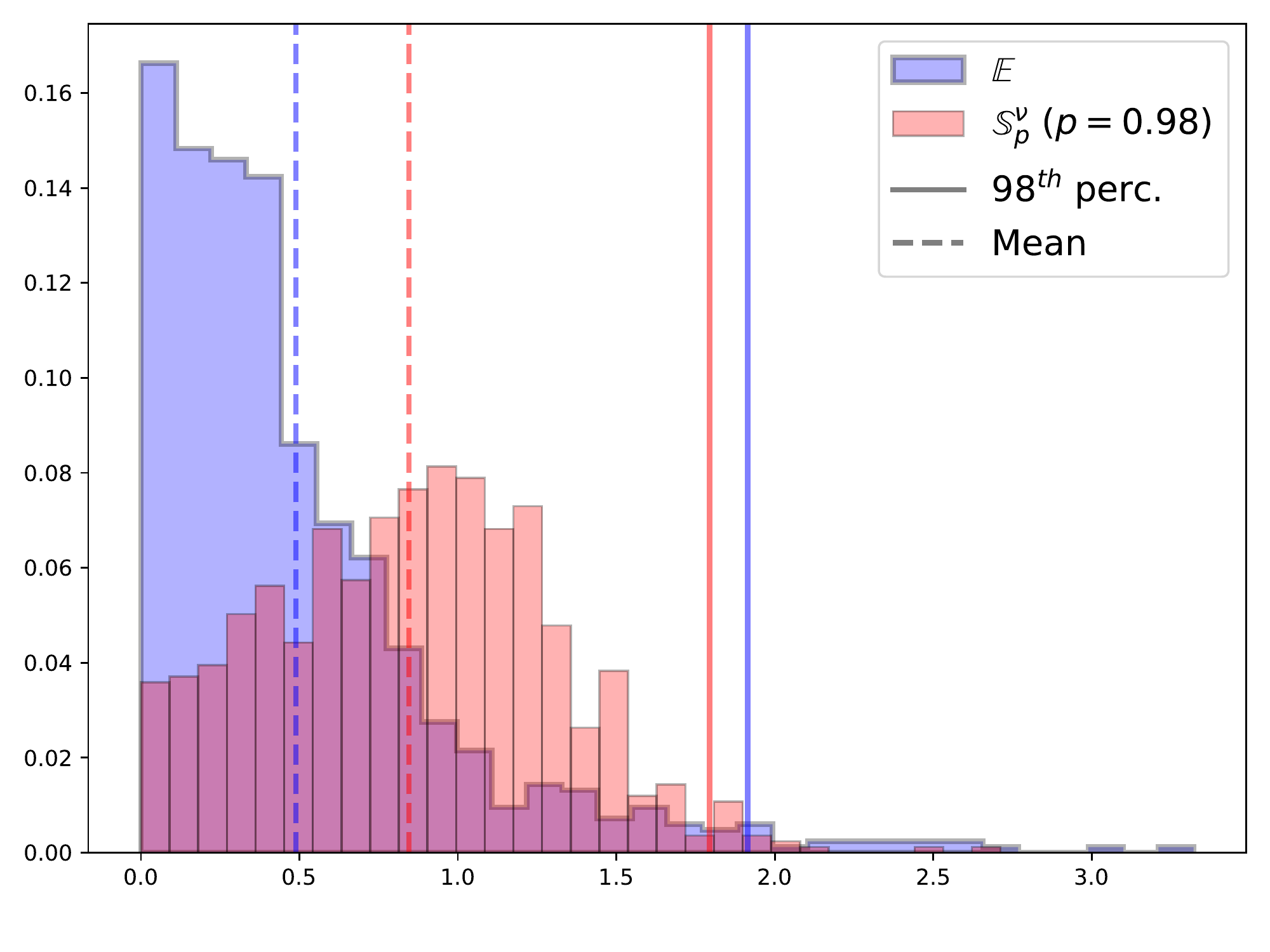}
\vspace*{-2ex}
  \caption{Regression: histogram of the regression errors on the testing dataset for the model learning by the superquantile approach (red) compared to the one of the classical empirical risk minimization (violet). We see a reshaping of the histogram of errors and a gain on worst-case errors.\label{fig:abalone}}
 \end{center}
\end{figure}

\medskip
\textbf{Regression and Least-Squares}

We consider a regularised least square regression on the dataset Abalone from the UCI Machine learning repository.
We perform a $80\%/20\%$ train-test split on the dataset.
We minimize the least-squares loss on the training set both in expectation and with respect to the superquantile (with $p=0.98$ and $\nu=0.1$).

We report on Figure~\ref{fig:abalone} the distribution of errors $|y_i - \trans{\x}x_i|$ for the testing dataset for both models $\x$ (standard in blue and superquantile in red). We observe that the superquantile model exhibits a thinner upper tail than the risk-neutral model, which is quantified by the shift to the left of $0.98$ quantile. This comes at the price of lower performance in expectations than the model trained with expectation, which is clear visible on the picture and quantified by the shift to the right of the mean.

\medskip
\textbf{Classification and Logistic regression}

We consider a logistic regression on the Australian Credit dataset. We randomly split the dataset with a $80\%/20\%$ train-test split for 5 different seeds. For each seed, we perform a pessimistic distributional shift on the training dataset by downsampling the majority class (similarly to what is done in \cite[Sec.\;5.2]{DBLP:conf/nips/CuriLJ020}). More precisely, we remove an important fraction of the majority class, randomly selected, so that it counts afterward for only $10\%$ of the minority class. We tune then the safety level parameter p by a k-cross validation on the shifted dataset and select the safety parameter yielding the best validation accuracy. The grid we use for tuning this parameter is [0.8, 0.85, 0.9, 0.95, 0.99]
We finally compute with this parameter the testing accuracy and the testing precision.

We report the testing accuracy and the testing precision averaged over the 5 different seeds on the table of Figure\;\ref{tab:table} with the associated standard deviation.
We observe that the superquantile model brings better performance for both in terms of accuracy and precision than the standard model.

\begin{figure}[h!]
\begin{center}
  \begin{tabular}{ccc} \hline
  Model & Accuracy & Precision  \\ \hline
  Standard & $0.65\pm 0.03$ & $0.56\pm 0.04$ \\
  Superquantile & $0.69\pm 0.04$ & $0.60\pm 0.05$\\ \hline
  \end{tabular}
 \vspace*{-1ex}
  \caption{Classification: better testing accuracy and precision for the superquantile approach, in the case of distributional shifts.\label{tab:table}}
  \end{center}
\end{figure}



\end{document}